\documentclass[12pt]{article}
\usepackage[T1]{fontenc}
\usepackage{lmodern,amsmath,amsthm,amsfonts,amssymb,graphicx,float,microtype,thmtools,underscore,mathtools}
\usepackage[shortlabels]{enumitem}
\setlist[itemize]{topsep=0ex,itemsep=0ex,parsep=0.3ex}
\setlist[enumerate]{topsep=0ex,itemsep=0ex,parsep=0.3ex}
\usepackage[usenames,dvipsnames,svgnames,table]{xcolor}
\usepackage{caption}
\usepackage{subcaption}
\usepackage[unicode=true]{hyperref}
\hypersetup{ 
colorlinks,
breaklinks,
linkcolor={blue!60!black},
citecolor={black},
urlcolor={blue!60!black},
pdftitle={Powers of planar graphs, blocking partitions and product structure}}
\usepackage[capitalise, compress, nameinlink, noabbrev]{cleveref}

\crefname{lem}{Lemma}{Lemmas}
\crefname{thm}{Theorem}{Theorems}
\crefname{open}{Open Problem}{Open Problems}
\crefname{cor}{Corollary}{Corollaries}
\crefname{obs}{Observation}{Observations}
\crefname{claim}{Claim}{Claims}
\crefname{prop}{Proposition}{Propositions}
\crefname{inv}{Invariant}{Invariants}
\crefformat{enumi}{#2#1#3}
\Crefformat{enumi}{Item #2#1#3}
\newcommand{\mathdefn}[1]{\textcolor{Maroon}{#1}}
\newcommand{\defn}[1]{\textcolor{Maroon}{\emph{#1}}}
\usepackage[longnamesfirst,numbers,sort&compress]{natbib}
\makeatletter
\def\NAT@spacechar{~}
\makeatother
\usepackage[margin=25mm]{geometry}
\renewcommand{\baselinestretch}{1.1}
\allowdisplaybreaks

\DeclarePairedDelimiter{\floor}{\lfloor}{\rfloor}
\DeclarePairedDelimiter{\ceil}{\lceil}{\rceil}
\DeclarePairedDelimiter{\abs}{\lvert}{\rvert}
 
\renewcommand{\epsilon}{\varepsilon}
\renewcommand{\emptyset}{\varnothing}
\renewcommand{\ge}{\geqslant}
\renewcommand{\le}{\leqslant}
\renewcommand{\geq}{\geqslant}
\renewcommand{\leq}{\leqslant}

\DeclareMathOperator{\dist}{dist}
\DeclareMathOperator{\mdist}{mdist}
\DeclareMathOperator{\tw}{tw}

\DeclareMathOperator{\out}{\text{out}}

\newcommand{\PP}{\mathcal{P}}

\newcommand{\GG}{\mathcal{G}}

\newcommand{\NN}{\mathbb{N}}

\newcommand{\RR}{\mathcal{R}}

\newcommand{\TT}{\mathcal{T}}

\newcommand{\WW}{\mathcal{W}}

\renewcommand{\thefootnote}{\fnsymbol{footnote}}
\theoremstyle{plain}
\newtheorem{thm}{Theorem}
\newtheorem{lem}[thm]{Lemma}
\newtheorem{cor}[thm]{Corollary}
\newtheorem{ques}[thm]{Question}
\newtheorem{prop}[thm]{Proposition}
\newtheorem*{prop*}{Proposition}
\newtheorem{obs}[thm]{Observation}
\newtheorem{claim}{Claim}
\newtheorem*{lem*}{Lemma}
\newtheorem{open}[thm]{Open Problem}
\theoremstyle{definition}

\newtheorem*{conj*}{Conjecture}
\theoremstyle{remark}
\newtheorem*{inv*}{Invariant}
\newcommand{\BlockingNumber}{222}
\newcommand{\TreewidthBound}{15\,288\,899}
\newcommand{\TreewidthBoundPlusOne}{15\,288\,900}
\newcommand{\GenusBlockingNumber}{894}
\newcommand{\GenusTreewidthBound}{963\,922\,179}
\newcommand{\GenusTreewidthBoundPlusOne}{963\,922\,180}

\begin{document}
\title{\bf Powers of planar graphs,\\
product structure, and blocking partitions\footnotemark[1]}

\footnotetext[1]{A preliminary version of this paper appeared in the \emph{Proceedings of the 12th European Conference on Combinatorics, Graph Theory and Applications} (EUROCOMB’23), \href{https://doi.org/10.5817/CZ.MUNI.EUROCOMB23-049
}{doi:10.5817/CZ.MUNI.EUROCOMB23-049}.}

\author{
Marc Distel\,\footnotemark[2] \qquad
Robert~Hickingbotham\,\footnotemark[2] \\ 
Micha\l{} T. Seweryn\footnotemark[3]\qquad
David~R.~Wood\,\footnotemark[2]
}

\footnotetext[2]{School of Mathematics, Monash University, Melbourne, Australia (\texttt{\{marc.distel,robert.hickingbotham, david.wood\}@monash.edu}). Research of M.D.\ and R.H.\ supported by Australian Government Research Training Program Scholarships. Research of D.W.\ supported by the Australian Research Council. }

\footnotetext[3]{
Computer Science Department, Université libre de Bruxelles, Brussels, Belgium (\texttt{michal.seweryn@ulb.be}).
Research of M.T.S.\ supported by a PDR grant
from the Belgian National Fund for Scientific
Research (FNRS).}

\sloppy
	
\maketitle

\begin{abstract}
We prove that the $k$-power of any planar graph $G$ is contained in $H\boxtimes P\boxtimes K_{f(\Delta(G),k)}$ for some graph $H$ with bounded treewidth, some path $P$, and some function $f$. This resolves an open problem of Ossona de Mendez. In fact, we prove a more general result in terms of shallow minors that implies similar results for many `beyond planar' graph classes, without dependence on $\Delta(G)$. For example, we prove that every $k$-planar graph is contained in $H\boxtimes P\boxtimes K_{f(k)}$ for some graph $H$ with bounded treewidth and some path $P$, and some function $f$. This resolves an open problem of Dujmovi\'c, Morin and Wood. We  generalise all these results for graphs of bounded Euler genus, still with an absolute bound on the treewidth.

At the heart of our proof is the following new concept of independent interest. An \defn{\(\ell\)-blocking partition} of a graph \(G\) is a partition of \(V(G)\) into connected sets such that every path of length greater than \(\ell\) in \(G\) contains at least two vertices in one part. We prove that for some constant \(\ell \ge 1\) every graph of Euler genus $g$ has an \(\ell\)-blocking partition with parts of size bounded by a function of \(\Delta(G)\) and \(g\). Motivated by this result, we study blocking partitions in their own right. We show that every graph \(G\) has a \(2\)-blocking partition with parts of size bounded by a function of \(\Delta(G)\) and \(\tw(G)\). On the other hand, we show that 4-regular graphs do not have $\ell$-blocking partitions with bounded size parts. 
\end{abstract}


\renewcommand{\thefootnote}{\arabic{footnote}}

\newpage
\section{Introduction}
\label{Introduction}

Graph product structure theory describes complicated graphs as subgraphs of strong products\footnote{The \defn{strong product} of graphs~$A$ and~$B$, denoted by~\defn{${A \boxtimes B}$}, is the graph with vertex-set~${V(A) \times V(B)}$, where distinct vertices ${(v,x),(w,y) \in V(A) \times V(B)}$ are adjacent if
	${v=w}$ and ${xy \in E(B)}$, or
	${x=y}$ and ${vw \in E(A)}$, or
	${vw \in E(A)}$ and~${xy \in E(B)}$. }  of simpler building blocks, which typically have bounded treewidth\footnote{Let \defn{$\tw(H)$} denote the treewidth of a graph $H$ (defined in \cref{Preliminaries}).}. 
For example, \citet{DJMMUW20} proved the following product structure theorem for planar graphs, where a graph~$H$ is \defn{contained} in a graph $G$ if $H$ is isomorphic to a subgraph of $G$. 

\begin{thm}[\cite{DJMMUW20}]\label{PGPST}
Every planar graph is contained in $H \boxtimes P \boxtimes K_3$ for some planar graph $H$ with $\tw(H)\leq 3$ and for some path $P$. 
\end{thm}

This result has been the key to solving several long-standing open problems about queue layouts~\citep{DJMMUW20}, nonrepetitive colourings~\citep{DEJWW20}, centred colourings~\citep{DFMS21}, adjacency labelling~\cite{EJM23,DEGJMM21}, twin-width~\cite{BKW,KPS23,JP22}, vertex ranking~\citep{BDJM}, and box dimension~\citep{DGLTU22}. \cref{PGPST} has been extended in various ways for graphs of bounded Euler genus~\citep{DJMMUW20,DHHW22,KPS23}, 
graphs excluding an apex minor~\citep{DJMMUW20,ISW,DHHJLMMRW}, 
graphs excluding an arbitrary minor~\citep{DJMMUW20,ISW,UTW}, 
graphs of bounded tree-width~\citep{UTW,DHHJLMMRW}, 
graphs of bounded path-width~\citep{DHJMMW24}, 
and for various non-minor-closed classes \citep{DMW23,HW24}. 

Many of these results show that for a particular graph class $\GG$ there are integers $t,c$ such that every graph in $\GG$ is contained in $H\boxtimes P\boxtimes K_c$ for some graph $H$ with treewidth $t$ and for some path $P$. Here the primary goal is to minimise $t$, where minimising $c$ is a secondary goal. This paper proves product structure theorems of this form for powers of planar graphs and for various beyond planar graph classes. The distinguishing feature of our results is that $\tw(H)$ is bounded by an absolute constant, instead of depending on a parameter defining $\GG$. This is important because in several applications of such product structure theorems, the main dependency is on $\tw(H)$; see \cref{SectionCentredColourings} for an example.

First consider powers of planar graphs. For $k\in\NN$, the \defn{$k$-power} of a graph $G$, denoted \defn{$G^k$}, is the graph with vertex-set $V(G)$, where $vw\in E(G^k)$ if and only if $\dist_G(v,w)\in\{1,\dots,k\}$. \citet{DMW23} proved that for every planar graph $G$ of maximum degree $\Delta$, the $k$-power $G^k$ is contained in $H\boxtimes P\boxtimes K_{6 \Delta^{k}(k^4+3k^2)}$ for some graph $H$ with $\tw(H)\leq \binom{k+3}{3}-1$ and some path $P$. Dependence on $\Delta$ is unavoidable since, for example, if $G$ is the complete $(\Delta-1)$-ary tree of height $k$, then $G^{2k}$ is a complete graph on roughly $(\Delta-1)^k$ vertices. \citet{OdM-Banff} asked whether this bound on $\tw(H)$ could be made independent of $k$. In particular:

\begin{ques}[{\protect\citep{OdM-Banff}}]
\label{Power}
Is there a constant $t$ and a function $f$ such that for every planar graph $G$ and $k\in\NN$, the $k$-power $G^k$ is contained in $H\boxtimes P\boxtimes K_{f(k,\Delta(G))}$ for some graph $H$ with $\tw(H)\leq t$ and for some path $P$?
\end{ques}

We resolve this question, in the following strong sense. For integers $k,d\geq 1$ and a graph $G$, let \defn{$G^k_d$} be the graph with vertex-set $V(G)$ where $vw\in E(G^k_d)$ whenever there is a $vw$-path $P$ in $G$ of length at most $k$ such that every internal vertex of $P$ has degree at most $d$ in \(G\). The following theorem answers \cref{Power} in the affirmative, since $G^k=G^k_{\Delta(G)}$. 

\begin{thm}
\label{kPowerPlanar}
There is a function $f$ such that for every planar graph $G$ and for any integers $k,d\geq 1$, the graph $G^k_d$ is contained in $H\boxtimes P\boxtimes K_{f(k,d)}$ for some graph $H$ with $\tw(H)\leq \TreewidthBound$ and for some path $P$.
\end{thm}

We chose to simplify the proof instead of optimising the constant upper bound on $\tw(H)$ in \cref{kPowerPlanar} and in our other results. 

\cref{kPowerPlanar} is in fact a corollary of a more general result expressed in terms of shallow minors. 

\subsection{Shallow Minors and Beyond Planar Graphs}

Let $G$ and $H$ be graphs and let $r,s\geq 0$ be integers.
$H$ is a \defn{minor} of $G$ if a graph isomorphic to $H$ can be obtained from $G$ by vertex deletion, edge deletion, and edge contraction. A class $\GG$ of graphs is \defn{minor-closed} if for every $G\in\GG$ every minor of $G$ is in $\GG$. A \defn{model} $(B_x\colon x\in V(H))$ of $H$ in $G$ is a collection of vertex-disjoint connected subgraphs in $G$ such that $B_x$ and $B_y$ are adjacent in $G$ for every edge $xy \in E(H)$. Clearly $H$ is a minor of $G$ if and only if $G$ contains a model of $H$. If there exists a model of $H$ in $G$ such that $B_x$ has radius at most $r$ for all $x \in V(H)$, then $H$ is an \defn{$r$-shallow minor} of $G$. A \defn{rooted model} $((B_x,v_x) \colon x\in V(H))$ of $H$ is a model of $H$ where each $B_x$ has a corresponding root $v_x\in V(B_x)$. If for every $x\in V(H)$ and for every $u\in V(B_x)\setminus \{v_x\}$, we have $\dist_{B_x}(v_x,u)\leq r$ and $\deg_{B_x}(u)\leq s$, then $((B_x,v_x)  \colon x\in V(H))$ is an \defn{$(r,s)$-shallow model} and $H$ is an \defn{$(r,s)$-shallow minor} of $G$. Clearly, if $H$ is an $r$-shallow minor of $G$, then $H$ is an $(r,\Delta(G))$-shallow minor of $G$. However, these definitions do not assume $\Delta(G)$ is bounded, since each vertex $v_x$ may have unbounded degree in $B_x$ and each vertex $u\in V(B_x)$ may have unbounded degree in $G$. 

Building on the work of \citet{DMW23}, \citet{HW24} showed that shallow minors inherit product structure. 

\begin{thm}[\cite{HW24}]
\label{HWShallowMinors}
If a graph $G$ is an $r$-shallow minor of $H \boxtimes P \boxtimes K_c$ where $\tw(H)\leq t$, then  $G$ is contained in $ J \boxtimes P \boxtimes K_{c(2r+1)^2}$ for some graph $J$ with $\tw(J)\leq \binom{2r+1+t}{t}-1$.
\end{thm}

Our main contribution is the following product structure theorem for $(r,s)$-shallow minors of planar graphs, where $J$ has treewidth bounded by an absolute constant. 

\begin{thm}
\label{ImprovedShallowMinorsPlanar}
There is a function $f$ such that for every planar graph $G$ and for every $(r,s)$-shallow minor $H$ of $G\boxtimes K_{d}$, $H$ is contained in $J \boxtimes P \boxtimes K_{f(d,r,s)}$ for some graph $J$ with $\tw(J)\leq \TreewidthBound$ and for some path $P$.
\end{thm}

\cref{ImprovedShallowMinorsPlanar} is useful since, as observed by \citet{HW24}, many non-minor-closed graph classes can be described as shallow minors of a strong product of a planar graph with a small complete graph. For example, for any graph $G$ with maximum degree $\Delta$, \citet{HW24} observed that $G^k$ is a $\floor{\tfrac{k}{2}}$-shallow minor of $G\boxtimes K_{\Delta^{\floor{k/2}+1}}$.
The proof is readily adapted\footnote{Let $D$ be the set of vertices with degree at most $d$ in $G$.  For each vertex $v\in V(G)$, let $B_v'$ be the subgraph of $G$ induced by the set of vertices $x\in V(G)$ for which there is a $vx$-path $P$ in $G$ of length at most $\floor{\frac{k}{2}}$ such that $V(P-v)\subseteq D$. So the radius of \(B_v'\) is at most $\floor{\frac{k}{2}}$ and there is an edge between \(V(B_u')\) and \(V(B_v')\) in \(G\) for each \(uv \in E(G^k)\). Furthermore, $|\{v\in V(G)\colon x\in V(B_v')\}|\leq d^0 + \cdots + d^{\floor{k/2}} \le d^{\floor{k/2}+1}$ for each $x\in V(G)$. 
So an arbitrary injective map from $\{v\in V(G)\colon x\in V(B_v')\}$ to $V(K_{d^{\floor{k/2}+1}})$ for each vertex $x\in V(G)$  defines a subgraph $B_v$ of $G\boxtimes K_{d^{\floor{k/2}+1}}$ such that the projection of $B_v$ onto $G$ is $B_v'$ and $V(B_v)\cap V(B_u)=\emptyset$ for all distinct $u,v\in V(G)$. So $(B_v\colon v\in V(G^k_d))$ defines a model of $G^k_d$ in $G\boxtimes K_{d^{\floor{k/2}}+1}$ where each $B_v$ has radius at most $\floor{\frac{k}{2}}$, as required. By construction, $G^k_d$ is in fact a $(\floor{\frac{k}{2}},d)$-shallow minor of $G\boxtimes K_{d^{\floor{k/2}+1}}$.} to show that 
$G^k_d$ is a $(\floor{\tfrac{k}{2}},d)$-shallow minor of $G\boxtimes K_{d^{\floor{k/2}+1}}$. Thus \cref{ImprovedShallowMinorsPlanar} implies \cref{kPowerPlanar}.

\cref{ImprovedShallowMinorsPlanar} can also be applied to several well-studied beyond planar graph classes, which we now introduce. See \citep{DLM19,HT20} for surveys on beyond planarity.

A graph $G$ is \defn{$k$-planar} if $G$ has a drawing in the plane in which each edge is involved in at most $k$ crossings, where no three edges cross at a single point; such graphs are widely studied, see \citep{PachToth97,DMW23,DEW17,DMW17} for example. \citet{DMW23} proved that every $k$-planar graph is contained in $H\boxtimes P\boxtimes K_{18k^2+48k+30}$ for some graph $H$ of treewidth $\binom{k+4}{3}-1$ and for some path $P$. \citet{DMW23} asked whether this bound on $\tw(H)$ could be made independent of $k$. In particular:

\begin{ques}[{\protect\citep{DMW23}}]
\label{kPlanar}
Is there a constant $t$ and a function $f$ such that every $k$-planar graph $G$ is contained in $H\boxtimes P\boxtimes K_{f(k)}$ for some graph $H$ with $\tw(H)\leq t$?
\end{ques}

\cref{ImprovedShallowMinorsPlanar} resolves this question.

\begin{cor}
\label{kPlanarTheorem}
There is a function $f$ such that every $k$-planar graph $G$ is contained in $H\boxtimes P\boxtimes K_{f(k)}$ for some graph $H$ with $\tw(H)\leq \TreewidthBound$. 
\end{cor}

\begin{proof}
\citet{HW24} observed that $G$ is a $\ceil{\tfrac{k}{2}}$-shallow minor of $H\boxtimes K_2$, where $H$ is the planar graph obtained from $G$ by adding a dummy vertex at each crossing point. A close inspection of their proof reveals that each branch set in the model of $G$ in $H\boxtimes K_2$ is a subdivided star rooted at the high degree vertex. So $G$ is a $(\ceil{\tfrac{k}{2}},2)$-shallow minor of $H\boxtimes K_2$. The claim then follows from \cref{ImprovedShallowMinorsPlanar}.
\end{proof}

 A \defn{string graph} is the intersection graph of a set of curves in the plane with no three curves meeting at a single point. Such graphs are widely studied; see \citep{FP10,Mat14,PachToth-DCG02,SS-JCSS04,FP12} for example. For an integer $\delta \geq 1$, if each curve is involved in at most $\delta$ intersections with other curves, then the corresponding string graph is called a \defn{$\delta$-string graph}. 

 \begin{cor}
 \label{StringProductStructure}
    There is a function $f$ such that every $\delta$-string graph $G$ is contained in $J \boxtimes P \boxtimes K_{f(\delta)}$ for some graph $J$ with $\tw(J)\leq \TreewidthBound$ and for some path $P$.
\end{cor}

\begin{proof}
\citet{HW24} observed that $G$ is a  $\floor{\frac{\delta}{2}}$-shallow minor of $H \boxtimes K_2$, where $H$ is the planar graph obtained by adding a dummy vertex at each intersection point of two curves (and possibly adding isolated vertices). A close inspection of their proof reveals that each branch set of the model of $G$ in $H\boxtimes K_2$ is a path. So $G$ is a $(\floor{\frac{\delta}{2}},2)$-shallow minor of $H\boxtimes K_2$. The claim then follows from \cref{ImprovedShallowMinorsPlanar}.
\end{proof}

The following graph class was introduced by \citet{ABKKS18}. A \defn{fan-bundling} of a graph $G$ is an indexed  set $\mathcal{E}=(\mathcal{E}_v:v\in V(G))$ where $\mathcal{E}_v$ is a partition of the set of edges in $G$ incident to $v$. Each element of $\mathcal{E}_v$ is called  a \defn{fan-bundle}. For a fan-bundling $\mathcal{E}$ of $G$, let $G_\mathcal{E}$
be the graph whose vertices are the vertices of $G$ and the bundles of $\mathcal{E}$, where $vB$ is an edge of $G_\mathcal{E}$ whenever $v\in V(G)$ and $B\in \mathcal{E}_v$, and $B_1B_2$ is an edge of $G_\mathcal{E}$ whenever $vw\in E(G) $ and  $vw\in B_1\in\mathcal{E}_v$ and $vw\in B_2\in\mathcal{E}_w$. A graph $G$ is \defn{$k$-fan-bundle planar} if for some fan-bundling $\mathcal{E}$ of $G$, the graph $G_\mathcal{E}$ has a drawing in the plane such that each edge $B_1B_2\in E(G_\mathcal{E})$ is in no crossings, and 
each edge $vB\in E(G_\mathcal{E})$ is in at most $k$ crossings.
 
\begin{cor}
\label{FanBundleProductStructure}
    There is a function $f$ such that every $k$-fan-bundle planar graph $G$ is contained in $J \boxtimes P \boxtimes K_{f(k)}$ for some graph $J$ with $\tw(J)\leq \TreewidthBound$ and for some path $P$.
\end{cor}

\begin{proof}
    \citet{HW24} showed that $G$ is a $(k+1)$-shallow minor of $H \boxtimes K_2$ for some planar graph $H$. A close inspection of their proof reveals that each branch set of the model of $G$ in $H\boxtimes K_2$ is a rooted subdivided star. So $G$ is a $(k+1,2)$-shallow minor of $H\boxtimes K_2$. The claim then follows from \cref{ImprovedShallowMinorsPlanar}.
\end{proof}

\subsection{Other Surfaces}
\label{SurfacesIntro}

We generalise all of the above results for graphs embeddable on any fixed surface as follows.  The \defn{Euler genus} of a surface with~$h$ handles and~$c$ cross-caps is~${2h+c}$. The \defn{Euler genus} of a graph~$G$ is the minimum integer $g\geq 0$ such that there is an embedding of~$G$ in a surface of Euler genus~$g$; see \cite{MoharThom} for background about graph embeddings in surfaces. \cref{ImprovedShallowMinorsPlanar} generalises as follows. 

\begin{thm}
\label{ImprovedShallowMinorsGenus}
There is a function $f$ such that for every graph $G$ of Euler genus $g$, every $(r,s)$-shallow minor $H$ of $G\boxtimes K_{d}$ is contained in $J \boxtimes P \boxtimes K_{f(d,r,s,g)}$ for some graph $J$ with $\tw(J)\leq \GenusTreewidthBound$.
\end{thm}

\cref{kPowerPlanar} generalises as follows.
The proof is directly analogous to the proof of \cref{kPowerPlanar}, using \cref{ImprovedShallowMinorsGenus}
instead of \cref{ImprovedShallowMinorsPlanar}.

\begin{cor}
\label{kPowerGenus}
There is a function $f$ such that for every graph $G$ of Euler genus $g$ and for any integers $k,d\geq 1$, the graph $G^k_d$ is contained in $H\boxtimes P\boxtimes K_{f(d,g,k)}$ for some graph $H$ with $\tw(H)\leq \GenusTreewidthBound$ and for some path $P$.
\end{cor}

We generalise \cref{kPlanarTheorem} as follows, where a graph $G$ is \defn{$(g,k)$-planar} if $G$ has a drawing in a surface of Euler genus $g$ in which each edge is involved in at most $k$ crossings, where no three edges cross at a single point. Such graphs are widely studied~\citep{GB07,DEW17,DMW17,DMW23}. 
\citet{DMW23} proved that every $(g,k)$-planar graph is contained in $H\boxtimes P\boxtimes K_{\max\{2g,3\} (6k^2+16k+10)}$, for some graph $H$ of treewidth $\binom{k+4}{3}-1$ and for some path $P$. We improve the treewidth bound to an absolute constant. The proof is directly analogous to the proof of \cref{kPlanarTheorem}.

\begin{cor}
\label{gkPlanarTheorem}
There is a function $f$ such that every $(g,k)$-planar graph $G$ is contained in $H\boxtimes P\boxtimes K_{f(g,k)}$ for some graph $H$ with $\tw(H)\leq \GenusTreewidthBound$. 
\end{cor}
 
We generalise \cref{StringProductStructure} as follows, where a
\defn{($g,\delta)$-string graph} is the intersection graph of a set of curves in a surface of Euler genus $g$, such that no three curves meet at a single point, and each curve is involved in at most $\delta$ intersections with other curves.
The proof of \cref{gStringProductStructure} is directly analogous to the proof of \cref{StringProductStructure}.

\begin{cor}
\label{gStringProductStructure}
There is a function $f$ such that every $(g,\delta)$-string graph $G$ is contained in $J \boxtimes P \boxtimes K_{f(\delta)}$ for some graph $J$ with $\tw(J)\leq \GenusTreewidthBound$ and for some path $P$.
\end{cor}

We generalise \cref{FanBundleProductStructure} as follows, where a graph $G$ is \defn{$(g,k)$-fan-bundle planar} if for some fan-bundling $\mathcal{E}$ of $G$, the graph $G_\mathcal{E}$ has a drawing in a surface of Euler genus $g$ such that each edge $B_1B_2\in E(G_\mathcal{E})$ is in no crossings, and each edge $vB\in E(G_\mathcal{E})$ is in at most $k$ crossings. The proof of \cref{gFanBundleProductStructure} is directly analogous to the proof of \cref{FanBundleProductStructure}.
 
\begin{cor}
\label{gFanBundleProductStructure}
There is a function $f$ such that every $(g,k)$-fan-bundle planar graph $G$ is contained in $J \boxtimes P \boxtimes K_{f(k)}$ for some graph $J$ with $\tw(J)\leq \GenusTreewidthBound$ and for some path $P$.
\end{cor}

\subsection{Application: Centred Colourings}
\label{SectionCentredColourings}

\citet{NesOdM-GradI} introduced the following definition. For an integer $p\geq 1$, a vertex colouring $\phi$ of a graph $G$ is \defn{$p$-centred }if, for every connected subgraph $X \subseteq G$, $|\{\phi(v):v \in V(X)\}|>p$ or there exists some $v \in V(X)$ such that $\phi(v)\neq \phi(w)$ for every $w\in V(X)\setminus\{v\}$. For an integer $p\geq 1$, the \defn{$p$-centred chromatic number} of a graph $G$, denoted by $\chi_p(G)$, is the minimum integer $c\geq 0$ such that $G$ has a $p$-centred $c$-colouring. Centred colourings are important within graph sparsity theory as they characterise graph classes with bounded expansion \cite{NesOdM-GradI}.

\citet{DFMS21} established that $\chi_p(G\boxtimes H)\leq \chi_p(G) \chi(H^p)$ for all graphs $G$ and $H$. \citet[Lemma~15]{PS21} showed that $\chi_p(G)\leq \binom{p+t}{t}$ for every graph $G$ with treewidth at most $t$. It follows that 
if $G \subseteq H \boxtimes P \boxtimes K_{\ell}$ and $\tw(H)\leq t$, then
\begin{equation}
\label{CenteredColouring}
\chi_p(G)\leq \ell (p+1) \chi_p(H) \leq 
\ell (p+1)\tbinom{p+t}{t}
\in O_\ell(p^{t+1}).
\end{equation}
Thus, \cref{kPowerPlanar,kPlanarTheorem,FanBundleProductStructure} imply:
\begin{itemize}
\item for every planar graph $G$ and any  integers $k,d\geq 1$,  $\chi_p(G^k_d)\in O_{k,d}(p^{\TreewidthBoundPlusOne})$;
\item for every $k$-planar graph $G$, $\chi_p(G)\in O_k(p^{\TreewidthBoundPlusOne})$;
\item for every $k$-fan-bundle graph $G$, $\chi_p(G)\in O_k(p^{\TreewidthBoundPlusOne})$.
\end{itemize}
Similarly, \cref{kPowerGenus,gkPlanarTheorem,gFanBundleProductStructure} imply:
\begin{itemize}
\item for every graph $G$ of Euler genus $g$ and for any integers $k,d\geq 1$, $\chi_p(G^k_d)\in O_{g,k,d}(p^{\GenusTreewidthBoundPlusOne})$;
\item for every $(g,k)$-planar graph $G$, $\chi_p(G)\in O_{g,k}(p^{\GenusTreewidthBoundPlusOne})$;
\item for every $(g,k)$-fan-bundle graph $G$, $\chi_p(G)\in O_{g,k}(p^{\GenusTreewidthBoundPlusOne})$.
\end{itemize}
For $k$-planar or $(g,k)$-planar graphs $G$, the best previously known bound was $\chi_p(G)\in O_{g,k}(p^{\binom{k+4}{3}})$, due to \citet{DMW23}. The above results significantly improve this bound (for large $k$). 

\subsection{Paper Outline}

It remains to prove \cref{ImprovedShallowMinorsPlanar,ImprovedShallowMinorsGenus}. The proofs of these results depend on the notion of a `blocking partition', which we believe is of independent interest. Following a section of preliminary definitions, \cref{KeyResult} introduces
and states our main results about blocking partitions: \cref{BlockingPlanar} for planar graphs and \cref{BlockingGenus} for graphs of Euler genus $g$. We then show how \cref{BlockingPlanar,BlockingGenus} imply  \cref{ImprovedShallowMinorsPlanar,ImprovedShallowMinorsGenus}. \cref{BlockingPlanar} is the heart of the paper, and is proved in \cref{ChordalPartitionSection,Refinement,Analysis}.  \cref{BlockingGenus} is then proved in \cref{Surfaces} as a corollary of \cref{BlockingPlanar}. 
\cref{Reflections} considers which graph classes admit blocking partitions of bounded width. We show that bounded maximum degree is necessary but not sufficient, and that bounded maximum degree and bounded treewidth are sufficient. \cref{OpenProblems} concludes by introducing some natural open problems that arise from this work. 

\section{Preliminaries}
\label{Preliminaries}

We consider simple, finite, undirected graphs~$G$ with vertex-set~${V(G)}$ and edge-set~${E(G)}$. See \citep{Diestel5} for graph-theoretic definitions not given here. Let~${\NN \coloneqq \{1,2,\dots\}}$ and~${\NN_0 \coloneqq \{0,1,\dots\}}$. A \defn{graph class} is a collection of graphs closed under isomorphism. 

We use the following notation for a graph $G$. For $v\in V(G)$, let $\mathdefn{N_G(v)}:=\{w\in V(G):vw\in E(G)\}$ and $\mathdefn{N_G[v]}:=N_G(v)\cup\{v\}$. For $S\subseteq V(G)$, let
$\mathdefn{N_G(S)}:=\bigcup_{v\in S}N_G(v)\setminus S$.

The \defn{length} of a path $P$ is the number of edges in $P$. 
Given a graph \(G\) and two subsets \(A, B \subseteq V(G)\),
a path \(P\) in \(G\)
is an \defn{\(A\)--\(B\) path} if
either \(P\) consists of only one vertex \(x \in A \cap B\),
or \(P\) has length at least \(1\), one end of \(P\) belongs to \(A\), the other belongs to \(B\), and no inner vertex belongs to \(A \cup B\).
For vertices \(x, y \in V(G)\), an \defn{\(x\)--\(y\) path} is an
\(\{x\}\)--\(\{y\}\) path.
For a tree \(T\) and \(x, y \in V(T)\), we denote the unique
\(x\)--\(y\) path in \(T\) by \(x T y\).

For two subsets \(U_1, U_2 \subseteq V(G)\), let \defn{\(\dist_G(U_1, U_2)\)} denote the \defn{distance} between \(U_1\) and \(U_2\) in $G$; that is, the length of a shortest
\(U_1\)--\(U_2\) path in $G$ (or \(+\infty\) if no such path exists).
In this notation, the role of  \(U_i\) can be played by a vertex \(x\), which is then interpreted as the singleton \(\{x\}\); for example, we write \(\dist_G(x, U)\) rather than \(\dist_G(\{x\}, U)\). Similarly, the role of \(U_i\) can be played by an edge \(x_1 x_2 \in E(G)\), which is then interpreted as the set \(\{x_1, x_2\}\), or by a set of edges \(M \subseteq E(G)\) which is interpreted as \(\bigcup_{xy\in M}\{x, y\}\). A path \(P\) in a graph \(G\) is \defn{geodesic} if it is a shortest path between its ends in \(G\), which implies \(\dist_P(x, y) = \dist_G(x, y)\) for any \(x, y \in V(P)\).

In a plane embedding of a graph $G$, a \defn{face} is a connected component of \(\mathbb{R}^2 - G\). We use \defn{closure} and \defn{boundary} in the topological sense. So the closure of a face $f$ is the union of $f$ and the boundary of $f$.
        
A \defn{tree-decomposition} of a graph~$G$ is a collection~${\WW = (W_x \colon x \in V(T))}$ of subsets of~${V(G)}$ indexed by the nodes of a tree~$T$ such that
(a) for every edge~${vw \in E(G)}$, there exists a node~${x \in V(T)}$ with~${v,w \in W_x}$; and 
(b) for every vertex~${v \in V(G)}$, the set~${\{ x \in V(T) \colon v \in W_x \}}$ induces a non-empty (connected) subtree of~$T$. 
Each set~$W_x$ in~$\WW$ is called a \defn{bag}. 
The \defn{width} of~$\WW$ is~${\max\{ \abs{W_x} \colon x \in V(T) \}-1}$. 
The \defn{treewidth~$\tw(G)$} of a graph~$G$ is the minimum width of a tree-decomposition of~$G$. 
Treewidth is the standard measure of how similar a graph is to a tree. 
Indeed, a connected graph has treewidth at most 1 if and only if it is a tree.

Let $G$ and $H$ be graphs. A \defn{partition} of $G$ is a collection $\PP$ of sets of vertices in $G$ such that each vertex of $G$ is in exactly one element of $\PP$. Each element of $\PP$ is called a \defn{part}. Empty parts are allowed. The \defn{width} of $\PP$ is the maximum number of vertices in a part. The \defn{quotient} of $\PP$ (with respect to $G$) is the graph, denoted by \defn{$G/\PP$}, whose vertices are the non-empty parts of $\PP$, where distinct non-empty parts $A,B\in \PP$ are adjacent in $G/\PP$ if and only if some vertex in $A$ is adjacent in $G$ to some vertex in $B$. The quotient is defined analogously when \(\PP\) is a set of vertex-disjoint subgraphs of \(G\) whose vertex-sets partition \(G\). Then the vertices of \(G/\PP\) are subgraphs of \(G\) instead of sets of vertices.
An \defn{$H$-partition} of $G$ is a partition $\PP$ of $G$ such that $G/\PP$ is contained in $H$. The following observation connects partitions and products.

\begin{obs}[\citep{DJMMUW20}]
\label{ObsPartitionProduct}
For all graphs $G$ and $H$ and any integer $p\geq 1$, $G$ is contained in $H\boxtimes K_p$ if and only if $G$ has an $H$-partition with width at most $p$.
\end{obs}

A partition of a graph $G$ is \defn{connected} if the subgraph induced by each part is connected. In this case, the quotient is the minor of $G$ obtained by contracting each part into a single vertex.

A partition $\PP$ of $G$ is \defn{chordal} if $\PP$ is connected and $G/\PP$ is \defn{chordal}. 

A \defn{tree-partition} is a $T$-partition for some tree $T$. Such a $T$-partition is \defn{rooted} if $T$ is rooted. 

Let $G$ and $H$ be graphs and let $Z$ be a subgraph of $G\boxtimes H$. The \defn{projection} of $Z$ onto $G$ is the subgraph $Z'$ of $G$ where $V(Z') :=\{v\in V(G)\colon \text{$(v,x)\in V(Z)$ for some $x\in V(H)$}\}$ and $E(Z'):=\{uv\in E(G)\colon \text{$(u,x)(v,y)\in E(Z)$ for some $x,y\in V(H)$} \}$. 

A \defn{\textsc{bfs}-layering} of a connected graph $G$ is an ordered partition $(V_0,V_1,\dots)$ of $V(G)$ where $V_0 = \{r\}$ for some vertex $r\in V(G)$ and $V_i=\{v\in V(G):\dist_G(v,r)=i\}$ for each $i\geq 1$. A path $P$ is \defn{vertical} with respect to $(V_0,V_1,\dots)$ if $|V(P)\cap V_i|\leq 1$ for all $i\geq 0$. Let $T$ be a spanning tree of $G$, where for each non-root vertex $v\in V_i$ there is a unique edge $vw$ in $T$ for some $w\in V_{i-1}$. Then $T$ is called a \defn{\textsc{bfs}-spanning tree} of $G$. 

\section{Blocking Partitions}
\label{KeyResult}

Let $G$ be a graph and $\RR$ be a connected partition of $V(G)$. A path $P$ in $G$ is \defn{$\RR$-clean} if $|V(P)\cap V|\leq 1$ for each part $V\in \RR$. We say that $\RR$ is  \defn{$\ell$-blocking} if every $\RR$-clean path in $G$ has length at most $\ell$, as illustrated in \cref{BlockingGrid}. 
\begin{figure}[h]
    \centering
    \includegraphics{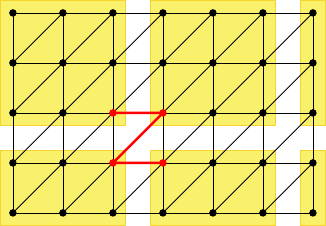}
    \caption{A \(3\)-blocking partition \(\RR\) of width \(9\). The red path is a longest \(\RR\)-clean path.}
    \label{BlockingGrid}
\end{figure}

The following result is the heart of this paper.

\begin{thm}\label{BlockingPlanar}
Every planar graph \(G\) with maximum degree $\Delta$ has a $\BlockingNumber$-blocking partition $\RR$ with width at most
$f(\Delta):=10\Delta^{80}(3612\,\Delta^{452}+900)$.
\end{thm}

\cref{BlockingPlanar} is proved in \cref{ChordalPartitionSection,Refinement,Analysis}. 
\cref{Surfaces} proves the following extension of \cref{BlockingPlanar}.

\begin{thm}\label{BlockingGenus}
Every graph $G$ with Euler genus $g$ and maximum degree $\Delta$ has a $\GenusBlockingNumber$-blocking partition with width at most $f(\Delta,g):=\max\{10\Delta^{80}(3612\Delta^{452}+900),8950g^2 + 1796g\}$.
\end{thm}

To show that \cref{BlockingPlanar,BlockingGenus}  imply our main results (\cref{ImprovedShallowMinorsPlanar,ImprovedShallowMinorsGenus}), we use the following lemma.

\begin{lem}
\label{MinorBlockingShallowProduct}
Let $\GG$ be a minor-closed class such that for some function $f$ and integers $\ell,t,c\geq 1$,  
\begin{itemize}
\item every graph \(G\in \GG\) has an $\ell$-blocking partition $\RR$ with width at most $f(\Delta(G))$;
\item  every graph $G\in\GG$ is contained in $H \boxtimes P \boxtimes K_c$  for some graph $H$ with $\tw(H)\leq t$ and for some path $P$. 
\end{itemize}
Then there is a function $g$ such that for any  integers $r\geq 0$ and $d,s\geq 1$, for every graph $G\in\GG$, every $(r,s)$-shallow minor of $G\boxtimes K_{d}$ is contained in $J \boxtimes P \boxtimes K_{g(d,r,s,\ell,c)}$ for some graph $J$ with $\tw(J)\leq \binom{2\ell+5+t}{t}-1$ and for some path $P$.
\end{lem}

For planar graphs, \cref{MinorBlockingShallowProduct} is applicable with $\ell=\BlockingNumber$ by \cref{BlockingPlanar} and with $t=c=3$ by \cref{PGPST}. \cref{MinorBlockingShallowProduct} thus proves \cref{ImprovedShallowMinorsPlanar} since $\tw(J)\leq \binom{2\ell+5+t}{t}-1=\binom{2\cdot \BlockingNumber+5+3}{3}-1=\TreewidthBound$. 
For graphs of Euler genus $g$, \cref{MinorBlockingShallowProduct} is applicable with $\ell=\GenusBlockingNumber$ by \cref{BlockingGenus} and with $t=3$ and $c=\max\{2g,3\}$ by a result of \citet{DHHW22}. \cref{MinorBlockingShallowProduct} thus proves \cref{ImprovedShallowMinorsGenus} since $\tw(J)\leq \binom{2\ell+5+t}{t}-1=\binom{2\cdot \GenusBlockingNumber+5+3}{3}-1=\GenusTreewidthBound$.

We now work towards proving \cref{MinorBlockingShallowProduct}.

\begin{lem}\label{ShallowMinorsStep}
Let $\GG$ be a minor-closed class such that, for some function $f$ and integer $\ell\geq 1$,  every graph $G_0\in \GG$ has an $\ell$-blocking partition $\RR$ with width at most $f(\Delta(G_0))$. 
Then for any integers \(r > \ell + 2\) and \(s,d \ge 1\), 
for every graph $G\in \GG$, every \((r, s)\)-shallow minor $H$ of \(G \boxtimes K_d\) is an \((r-1, s')\)-shallow minor of \(G' \boxtimes K_{d'}\), where \(G'\) is a minor of $G$ and is thus in $\GG$, and \(s' = (ds)^r\) and \(d' = d\cdot f(ds)\).
\end{lem}

\begin{proof}
Let $((B_x,v_x)\colon x\in V(H))$ be an $(r,s)$-shallow model of $H$ in $G\boxtimes K_d$. 
For each \(x \in V(H)\), let \(B_x'\) and \(v_x'\) be the projections of
\(B_x\) and \(v_x\), respectively, onto \(G\).
Observe that for each \(x \in V(H)\),
the maximum degree of each \(B_x - v_x\) is at most
\(s\) and each vertex in \(B_x'\) is at distance at most \(r\)
from \(v_x'\).
Let \(G_0 := \bigcup (B_x' - v_x' : x \in V(H))\),
which is a subgraph of $G$ and therefore in $\GG$. Since every vertex in $G_0$ has at most $d$ vertices mapped to it, the maximum degree of
\(G_0\) is at most $ds$.
By assumption, there is an $\ell$-blocking partition $\RR$ of $G_0$ with width at most $f(ds)$.


Let \(\RR' := \RR \cup \{\{v\} \colon v \in V(G) \setminus V(G_0)\}\), which is a partition of $G$. Define $G':=G/\RR'$.
Since $\RR'$ is a connected partition, $G'$ is a minor of $G$ and is therefore in $\GG$. 
The width of \(\RR'\) is at most \(f(ds)\),
so \(G\) is contained in \(G' \boxtimes K_{f(ds)}\) by \cref{ObsPartitionProduct}.
By slightly abusing the notation, we identify the graph \(G\)
with the isomorphic subgraph of \(G' \boxtimes K_{f(ds)}\).
So the graphs \(B_x'\) are subgraphs of \(G' \boxtimes K_{f(ds)}\),
and each vertex of \(G' \boxtimes K_{f(ds)}\)
belongs to at most \(d\) graphs \(B_x'\).

For each \(x\in V(H)\), let \(T_x'\) be a \textsc{bfs}-spanning tree of
\(B_x'\) rooted at \(v_x'\).
Hence, the maximum degree of \(T_x' - v_x'\) is at most \(ds\)
and each vertex is at distance at most \(r\) from the root \(v_x'\) in \(T_x'\). 
So each component of \(T_x' - v_x'\) has at most \((ds)^0 + \ldots + (ds)^{r-1} < (ds)^r\) vertices'. Let \(\overline{T_x'}\) denote the graph obtained from
\(T_x'\) by adding each edge of
\(G' \boxtimes K_{f(ds)}\) that joins a vertex of \(T_x'\) 
to one of its descendants. Then \(\overline{T_x'} - v_x'\) has maximum degree at most $(ds)^r$.

Below we show that the maximum degree of \(\overline{T_x'} - v_x'\)
is at most \(s'\)
and each vertex in \(\overline{T_x'}\) is at distance at most
\(r-1\) from \(v_x'\). This implies that \(H\) is an \((r-1, s')\)-shallow
minor of \(G' \boxtimes K_{f(ds)} \boxtimes K_d\),
where an appropriate model
can be defined by choosing for each \(v \in V(G' \boxtimes K_{f(ds)})\) an injective map
from \(\{x \in V(H): v \in V(B_x')\}\) to \(V(K_d)\). 
Since \(G' \boxtimes K_{f(ds)} \boxtimes K_d\) is isomorphic
to \(G' \boxtimes K_{d'}\), the lemma will follow.

First we estimate the maximum degree of \(\overline{T_x'} - v_x'\).
Consider a vertex \(v \in V(T_x')\) at distance \(i \ge 1\) from the root \(v_x'\) in \(T_x'\). Then, for each
\(j \in \{1, \ldots, i - 1\}\), the vertex \(v\) has only one ancestor at distance \(j\) from \(v_x'\) in \(T_x'\).
Since the maximum degree of \(T_x' - v_x'\) is at most \(ds\),
for each \(j \in \{i+1, \ldots, r\}\), there are at most
\((ds)^{j-i}\) descendants of \(v\) at distance \(j\) from
\(v_x'\).
Therefore, \(v\) has at most \((ds)^{j-1}\)
neighbours in \(\overline{T_x'} - v_x'\) which are at distance
\(j\) from \(v_x'\) in \(T_x'\).
Hence, the degree of \(v\) in \(\overline{T_x'} - v_x'\)
is at most \((ds)^0 + \ldots + (ds)^{r-1} < (ds)^r\),
so the maximum degree of \(\overline{T_x'} - v_x'\) is at most
\((ds)^r\).

It remains to show that in each \(\overline{T_x'}\),
every vertex is at distance at most \(r-1\) from \(v_x'\).
Suppose to the contrary that some vertex \(u\) is at distance
at least \(r\) from \(v_x'\) in \(\overline{T_x'}\).
Since \(T_x' \subseteq \overline{T_x'}\), and
in \(T_x'\) every vertex is at distance at most \(r\) from \(v_x'\),
the vertex \(u\) must be at distance exactly
\(r\) from \(v_x'\) in \(T_x'\) and \(\overline{T_x'}\).
Let \(P=(u_0, \ldots, u_r)\) be the unique path between
\(v_x'\) and \(u\) in \(T_x'\) where \(u_0 = v_x'\) and
\(u_r = u\). Let $P'=(x_1, \ldots, x_r)$ be the projection of $P$ onto $G_0$. Then $P'$ is a path in \(G_0\) with length \(r-1 \ge \ell+1\), so it contains two vertices \(x_\alpha\) and \(x_\beta\) with \(1 \le \alpha < \beta\) that belong to the same part in \(\RR\). Thus the projection of \(u_\alpha\) and \(u_\beta\) (in  $G' \boxtimes K_{f(ds)}$) are the same vertex in \(G'\) and so, by the definition of the strong product, $u_\beta u_{\alpha-1}\in E(G' \boxtimes K_{f(ds)})$. Hence the distance between \(v_x'\) and \(u\) in \(\overline{T_x'}\) is less than \(r\), a contradiction.
\end{proof}

We prove the next lemma by iteratively applying \cref{ShallowMinorsStep}.


\begin{lem}
\label{ShallowMinorsIntermediate}
Let $\GG$ be a minor-closed class such that, for some function $f$ and integer $\ell\geq 1$,  every graph \(G\in \GG\) has an $\ell$-blocking partition with width at most  $f(\Delta(G))$. Then there is a function $h$ such that for any  integers \(r \geq 0\) and \(s,d \ge 1\), for every graph $G\in \GG $, every $(r,s)$-shallow minor $H$ of $G\boxtimes K_d$ is an $(\ell+2)$-shallow minor of  $Q \boxtimes K_{h(d,r,s,\ell)}$ for some minor $Q$ of $G$.
\end{lem}

\begin{proof}
   If $r\leq\ell+2$ then $H$ is an $(\ell+2)$-shallow minor of $Q \boxtimes K_{h(d,r,s,\ell)}$, where $Q=G$ and $h(d,r,s,\ell)=d$, and we are done. Now assume that $r>\ell+2$. Thus $r-\ell-2\geq 1$. 
Let $d_0:=d$ and $s_0:=s$.
Iteratively applying  \cref{ShallowMinorsStep}, we obtain a sequence $G_1,G_2,\dots,G_{r-\ell-2}$ of 
minors of $G$, such that for each $i\in\{1,\dots,r-\ell-2\}$, $H$ is an \((r-i, s_i)\)-shallow minor of \(G_i \boxtimes K_{d_i}\), where \(s_i = (d_{i-1}s_{i-1})^{r-i+1}\) and \(d_i = d_{i-1}\cdot f(d_{i-1}s_{i-1})\).
In particular (with $i=r-\ell-2$), 
$H$ is an \((\ell+2)\)-shallow minor of \(G_{r-\ell-2} \boxtimes K_{d_{r+\ell-2}}\). The result follows with $Q:=G_{r-\ell-2}$ and $h(d,r,s,\ell):=d_{r-\ell-2}$.
\end{proof}

\begin{proof}[Proof of \cref{MinorBlockingShallowProduct}] 
Let $G\in \GG$ and let $G'$ be an $(r,s)$-shallow minor of $G\boxtimes K_{d}$. By \cref{ShallowMinorsIntermediate}, 
$G'$ is an $(\ell+2)$-shallow minor of  $Q \boxtimes K_{h(d,r,s,\ell)}$ for some minor $Q$ of $G$. Thus $Q\in \GG$. By assumption, $Q$ is contained in $H \boxtimes P \boxtimes K_c$  for some graph $H$ with $\tw(H)\leq t$. Hence 
$G'$ is an $(\ell+2)$-shallow minor of 
$H \boxtimes P \boxtimes K_{c\,h(d,r,s,\ell)}$. 
By \cref{HWShallowMinors}, $G'$ is contained in $ J \boxtimes P \boxtimes K_{c(2(\ell+2)+1)^2\cdot g(d,r,s,\ell)}$ for some graph $J$ with $\tw(J)\leq \binom{2(\ell+2)+1+t}{t}-1$. The result follows with 
$g(d,r,s,\ell,c) := c(2(\ell+2)+1)^2\cdot h(d,r,s,\ell)$. 
\end{proof}

\section{The Chordal Partition}
\label{ChordalPartitionSection}

Our focus now is the proof of \cref{BlockingPlanar}, which is inspired by the construction of a chordal partition of a planar triangulation by \citet{HOQRS17}. They showed that every planar triangulation $G$ has a partition $\PP$ 
into paths \(P_1, \ldots, P_n\), such that for each \(i \in \{1, \ldots, n - 1\}\), the path \(P_{i+1}\) is geodesic in \(G - (V(P_1) \cup \cdots \cup V(P_i))\), 
\(P_{i+1}\) is adjacent to at most two of the paths \(P_1, \ldots, P_i\), and if $P_{i+1}$ is adjacent to $P_a$ and $P_b$ with $1\leq a<b\leq i$, then $P_a$ is adjacent to $P_b$. Then the quotient $G/\PP$ is chordal with treewidth 2.

Our \(\ell\)-blocking partition of a planar graph \(G\) will be obtained from a partition of $G$ into subtrees \(T_1, \ldots, T_n\) with similar properties:
for each \(i \in \{1, \ldots, n - 1\}\),
the tree \(T_{i+1}\) is adjacent to at most two
of the trees \(T_1, \ldots, T_i\),
and if \(T_{i+1}\) is adjacent to two of those trees, then they are adjacent to each other.
The final partition is then obtained by appropriately breaking each \(V(T_i)\)
into connected parts of bounded size.

Fix a planar graph \(G\) of maximum degree \(\Delta\)
and any planar embedding of \(G\).
This section constructs a \(6\)-blocking%
\footnote{Actually, the partition is \(4\)-blocking, but for simplicity we prove a weaker bound.}
chordal partition \(\TT\)
of \(G\).
Later sections refine this partition into a connected
(non-chordal) partition \(\RR\) with 
width bounded in terms of \(\Delta\),
and show that \(\RR\) is \(222\)-blocking, which will prove \cref{BlockingPlanar}.
Since \cref{BlockingPlanar} is trivial when \(\Delta \leq 2\), we assume that
\(\Delta \ge 3\).

Our construction of the
\(6\)-blocking
chordal partition is parameterised by
a positive integer \(\tau\). Ultimately, we 
will fix \(\tau = 37\),
but it will be easier to visualise the
construction for smaller values of \(\tau\).

We use the notion of \(F\)-bridges, as illustrated in \cref{fig:Bridges}. For a subgraph $F$ of $G$, an
\defn{\(F\)-bridge} is either a length-\(1\) path in \(G\)
that is edge-disjoint from \(F\) and is between two vertices in \(V(F)\) 
(such an \(F\)-bridge is \defn{trivial}), 
or a graph obtained
from a component \(C\) of \(G - V(F)\) by adding all vertices
in \(N_G(V(C))\) and all edges of \(G\) between \(V(C)\) and
\(N_G(V(C))\) (such an \(F\)-bridge is \defn{non-trivial}).
Observe that each edge of \(G\) outside 
\(F\) belongs to exactly one
\(F\)-bridge. In an \(F\)-bridge \(B\),
the set \(V(B) \cap V(F)\) is
the \defn{attachment-set}, and its 
elements are the \defn{attachment-vertices} of \(B\). Hence, if \(B\) is non-trivial
with attachment-set \(A\), then
\(B - A\) is a component of
\(G - V(F)\).

\begin{figure}[!ht]
    \centering
    \includegraphics{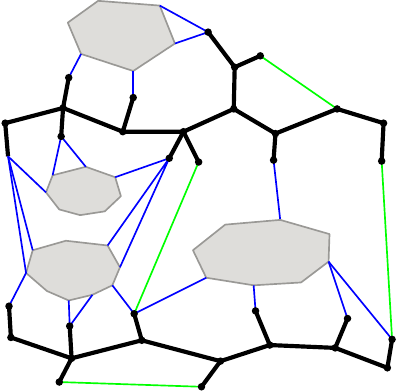}
    \caption{A graph \(G\)
    with a distinguished sub-forest \(F\) with two components (black). 
    There are four trivial
    \(F\)-bridges (green) and four non-trivial
    \(F\)-bridges (each of them is obtained from a component \(C\) of \(G - V(F)\) (gray) by adding all blue edges incident with a vertex of \(C\) (and their ends outside \(C\)))}
    \label{fig:Bridges}
\end{figure}

We will inductively define a sequence of vertex-disjoint trees
\(T_1, \ldots, T_m\) whose vertex-sets will form the chordal
partition of \(G\).
For each \(j \in \{0, \ldots, m\}\), we denote the forest
\(\bigcup_{i<j} T_i\) by \(F_j\)
(in particular, \(F_0\) is empty). For each
\(j \in \{0, \ldots, m\}\),
we maintain the following invariant. 
\begin{inv*}
  For every non-trivial \(F_j\)-bridge \(B\), the following hold:
  \begin{enumerate}[(\roman*)]
      \item \(B\) has attachment-vertices on at most two components of \(F_j\);
      \item for every component \(T_i\) of \(F_j\) that contains an attachment-vertex   of \(B\),
        the tree \(T_i\) is contained in the closure of the outer-face of \(B\), and the attachment-vertices of \(B\) in \(V(T_i)\) are leaves of \(T_i\);
      \item if \(B\) has attachment-vertices on two distinct components \(T_i\) and \(T_{i'}\) of \(F_j\),
        then \(T_i\) is contained in the closure of the outer-face of \(T_{i'} \cup B\), and 
        \(T_{i'}\) is contained in the closure of the outer-face of \(T_i \cup B\).
  \end{enumerate}
\end{inv*}

Our invariant implies the following.

\begin{claim}\label{Invariant}
    Suppose that the invariant is satisfied for some \(j \in \{0, \ldots, m\}\), and let
    \(B\) be a nontrivial \(F_j\)-bridge with at least one attachment-vertex.
    Let \(J\) be the union of \(B\) and all components \(T_i\) of \(F_j\) that contain
    an attachment-vertex of \(B\). Then, for each component \(T_i\) contained in \(J\),
    at least one and at most two attachment-vertices of \(B\) on \(T_i\) are on the boundary of the outer-face of \(J\).
\end{claim}
\begin{proof}
  By invariant (i), \(B\) has attachment-vertices on at most one component of \(F_j\) distinct from \(T_i\).
  By invariant (ii), \(T_i\) is contained in the closure of the outer-face of \(B\).
  Moreover, by invariant (iii), if \(B\) has an attachment-vertex on a tree \(T_{i'}\)
  distinct from \(T_i\), then \(T_i\) is contained in the closure of the outer-face of
  \(B \cup T_{i'}\).
  Therefore, in the facial walk along the outer-face of \(J\), the vertices and edges that belong to
  \(T_i\) appear consecutively, forming a (possibly closed) sub-walk \(W\).
  By invariant (ii), the attachment-vertices of \(B\) in \(V(T_i)\) are leaves of \(T_i\), 
  so only the terminal vertices of \(W\) are attachments of \(B\) in \(V(T_i)\)
  which lie on the boundary of the outer-face of \(J\).
  At most two vertices are terminal vertices of \(W\), so the claim holds.
\end{proof}



For \(j = 0\), the invariant is satisfied
because \(F_0\) is empty, so the \(F_0\)-bridges have no attachment-vertices.

Together with each tree \(T_j\) we will define a 
tuple \((B_j, A_j, X_j, A_j^{\out}, D_j,  T_j^0)\),
where \(T_j^0 \subseteq T_j \subseteq B_j \subseteq G\),
\(A_j^{\out} \subseteq A_j \subseteq V(F_{j-1})\),
\(X_j \subseteq \{1, \ldots, j-1\}\),
and \(D_j \subseteq V(T_j^0)\). 

Let \(j \ge 1\) be an integer, and  
suppose that we have already defined the trees \(T_1,\dots,T_{j-1}\), 
and thus the forest \(F_{j-1}\) is defined.
If \(V(F_{j-1}) = V(G)\), then terminate the construction
with a sequence of length \(j-1\).
Otherwise, let \(B_j\) be any
non-trivial \(F_{j-1}\)-bridge.
Let \(A_j\) denote the attachment-set of
\(B_j\), and let \(X_j\) denote the set of all
\(i \in \{1, \ldots, j-1\}\) such that \(B_j\)
has an attachment-vertex in \(V(T_i)\). By invariant (i),
we have \(|X_j| \le 2\). 

Let \(J := B_j \cup \bigcup_{i \in X_j} T_i\).
Define \(A_j^{\out}\) to be the set of attachment-vertices \(x \in A_j\) that lie on the boundary of the outer-face of \(J\).
By \cref{Invariant}, the set \(A_j^{\out}\) contains one or two vertices of each \(T_i\)
with \(i \in X_j\), so \(|A_j^{\out}| \le 4\) and
if \(A_j \neq \emptyset\), then \(A_j^{\out} \neq \emptyset\).

Define a non-empty subset
\(D_j \subseteq V(B_j - A_j)\) as follows.
If \(A_j = \emptyset\), then let \(D_j\)
be a set consisting of one arbitrary vertex
on the boundary of the outer-face of \(B_j\).
If \(A_j \neq \emptyset\) (and thus \(A_j^{\out} \neq \emptyset\)),
then let \(D_j\) denote the set of all vertices
\(x \in V(B_j - A_j)\) such that
\(\dist_G(x, A_j^{\out}) \le \tau\).
In \(B_j\), every vertex from \(A_j^{\out}\)
has a neighbour in \(V(B_j - A_j)\)
and such a neighbour belongs to \(D_j\)
(recall that \(\tau \ge 1\)). Hence, \(D_j\) is non-empty.

Let \(T_j^0\) be a tree in \(B_j - A_j\) that
contains all vertices in \(D_j\) and has the
smallest possible number of edges, and let
\(T_j\) be a tree obtained from \(T_j^0\)
by attaching each vertex
\(x \in N_{B_j - A_j}(V(T_j^0))\)
with
any edge of \(G\) between \(x\) and
\(V(T_j^0)\).
See \cref{FigChordalPartition}.
\begin{figure}
     \centering
     \includegraphics{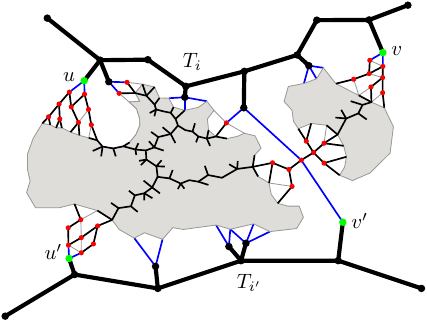}
        \caption{A possible situation in the construction of the tree \(T_j\) for \(\tau=3\), where \(B_j\)
        has attachment-vertices on the trees \(T_i\) and
        \(T_{i'}\).
        Here, \(A_j^{\out} = \{u, v, u', v'\}\), the vertices from \(D_j\) are red, and the tree \(T_j\) consists of the red and black edges .}
        \label{FigChordalPartition}
\end{figure}

We now verify that for such \(T_j\), the invariant is satisfied.
Let \(B\) be a non-trivial \(F_j\)-bridge.
If \(B\) has no attachment-vertex on \(T_j\), then \(B\) is an
\(F_{j-1}\)-bridge distinct from
\(B_j\), and the invariant is satisfied by induction.
Hence, we assume that \(B\) has an attachment-vertex
on \(T_j\). Since every component
of \(G - V(F_j)\) which is adjacent in \(G\)
to \(T_j\) is a component of
\((B_j - A_j) - V(T_j)\),
the \(F_j\)-bridge \(B\) is contained
in \(B_j\).
Note that every attachment-vertex of \(B\)
that is not on \(T_j\) lies on a tree
\(T_i\) with \(i < j\), and thus, is an
attachment-vertex of \(B_j\).

Suppose first that \(X_j = \emptyset\).
Then \(B\) has all its attachment-vertices on \(T_j\).
The only vertex \(x\) of \(D_j\) is on the boundary of the outer-face of \(B_j\).
The tree \(T_j\) is a star with a centre at \(x\) and whose leaves are the neighbours of \(x\)
in \(G\). Therefore \(B\) can intersect \(T_j\) only in its leaves, and \(T_j\)
is in the closure of the outer-face of \(B\), so the invariant holds.

Now suppose that \(X_j \neq \emptyset\),
and let \(i \in X_j\).
By induction, the tree \(T_i\) intersects the boundary of the
outer-face of \(J\), and we can write the facial 
walk along the outer-face of \(J\)
as \(W = v_0 e_0 v_1 e_1 \cdots e_{n-1} v_n\) where \(v_0 = v_n\) and for some
\(s \in \{0, \ldots, n-1\}\) we have \(V(W) \cap V(T_i) = \{v_0, \ldots, v_s\}\)
and \(E(W) \cap E(T_i) = \{e_0, \ldots, e_{s-1}\}\).
We have \(A_j^{\out} \cap V(T_i) = \{v_0, v_s\}\) (possibly \(v_0 = v_s\)).
Each of the edges \(e_{n-1} = v_0 v_{n-1}\) and \(e_s = v_s v_{s+1}\)
has an end in \(V(T_i)\) but does not
belong to \(T_i\). Hence both of
these edges are edges of \(B_j\), so
\(\{v_{n-1}, v_{s+1}\} \subseteq V(B_j - A_j)\).
Since \(\{v_0, v_s\} \subseteq A_j^{\out}\), we have
\(\{v_{n-1}, v_{s+1}\} \subseteq D_j \subseteq V(T_j^0)\).

We now show that \(B\) has attachment-vertices on at most two components of \(F_j\).
Every attachment-vertex of \(B\) that is not in $V(T_j)$, is an attachment-vertex of \(B_j\).
Hence, if \(X_j = \{i\}\), then \(B\)
can only have attachment-vertices
on \(T_j\) and \(T_i\).
Therefore, suppose that \(X_j = \{i, i'\}\) with \(i' \neq i\).
We need to show that \(B\) has an attachment-vertex on at most one of the trees \(T_i\) and \(T_{i'}\).
By our invariant, \(T_{i'}\) is in the closure of the outer-face of \(T_i \cup B\).
Therefore, \(T_{i'}\) intersects \(\{v_{s+2}, \ldots, v_{n-2}\}\).
Since the vertices \(v_0, \ldots, v_s\) belong to \(T_i\),
the path \(v_{n-1} T_j^0 v_{s+1}\) separates the trees \(T_i\) and \(T_{i'}\)
in \(J\). Since \(T_j^0 \subseteq T_j\), every component of \((B_j - A_j) - V(T_j)\)
is adjacent to at most one of the trees \(T_i\) and \(T_{i'}\).
Since \(B\) is a non-trivial \(F_j\)-bridge contained in \(B_j\),
this means that \(B\) has attachment-vertices in at most one of the trees \(T_i\) and \(T_{i'}\),
as required. Hence, \(B\) has attachment-vertices
on at most two components of \(F_j\).

Assume without loss of generality that
every attachment-vertex of $B$ that does not lie on $T_j$ belongs to $T_i$.

Next, we show that every attachment-vertex of
\(B\) is a leaf of \(T_i\) or \(T_j\).
Every attachment-vertex of \(B\) on \(T_i\)
is an attachment-vertex of \(B_j\), and by 
induction, is a leaf of \(T_i\).
Since the tree \(T_j\) was obtained from
\(T_j^0\) by attaching all adjacent 
vertices in $B_j-A_j$ as leaves, all attachment-vertices of
\(B\) on \(T_j\) belong to
\(V(T_j) \setminus V(T_j^0)\), and therefore are leaves of 
\(T_j\).

Finally, we argue that the tree
\(T_i\) is in the
closure of the outer-face of
\(T_j \cup B\), and the tree
\(T_j\) is in the
closure of the outer-face of
\(T_i \cup B\). This will imply
that the trees \(T_i\) and \(T_j\)
are in the closure of the outer-face of
\(B\), which will complete the proof
of the invariant.
By induction, the tree \(T_i\)
is in the closure of the outer-face of
\(B_j\). Since
\(T_j \cup B \subseteq B_j\), the tree
\(T_i\) is in the closure of the
outer-face of \(T_j \cup B\).
The vertex \(v_{s+1} \in V(T_j^0)\)
is on the boundary of the outer-face of
\(J\). Since \(T_j^0\) and
\(T_i \cup B\) are disjoint subgraphs of
\(J\), the tree \(T_j^0\) is on the
outer-face of \(T_i \cup B\).
The tree \(T_j\) is obtained from
\(T_j^0\) by attaching leaves, so
it is contained in the closure
of the outer-face of \(T_i \cup B\).
This completes the proof of the invariant
and the inductive construction.

From now on, we assume that \(\TT = \{T_1, \ldots, T_m\}\) is a fixed partition obtained by our construction for \(\tau = 37\), 
with a tuple \((B_j, A_j, X_j, A_j^{\out}, D_j, T_j^0)\) associated to each tree \(T_j\).

For later reference,
we make explicit some implications
of the inductive construction.

\begin{claim}\label{VTj0Separates}
    For each \(j \in \{1, \ldots, m\}\),
    if \(X_j = \{i, i'\}\) with
    \(i \neq i'\),
    then the tree \(T_j^0\) separates
    the trees \(T_i\) and \(T_{i'}\)
    in the graph
    \(T_i \cup T_{i'} \cup B_j\).
    Consequently, 
    the tree \(T_j^0\) separates the sets
    \(A_j \cap V(T_i)\) and \(A_j \cap V(T_{i'})\) in the graph \(B_j\).
\end{claim}

\begin{claim}\label{NoAttachmentOnTj0}
    For each \(j \in \{1, \ldots, m\}\),
    no non-trivial \(F_j\)-bridge has
    an attachment-vertex on \(T_j^0\).
\end{claim}

\begin{claim}\label{TiAdjacentToTj}
    For any \(j \in \{1, \ldots, m\}\)
    and \(i \in X_j\),
    the graph \(B_j\) contains an edge
    between \(A_j^{\out} \cap V(T_i)\)
    and \(D_j\). In particular,
    \(T_i\) is adjacent to \(T_j\)
    in \(G\).
\end{claim}

We will also use the following simple
properties of our construction.
\begin{claim}\label{FjBridgeInBj}
    For \(j \in \{1, \ldots, m\}\),
    an \(F_j\)-bridge \(B\)
    has an attachment-vertex on \(T_j\)
    if and only if \(B \subseteq B_j\).
\end{claim}
\begin{proof}
    Suppose that \(B\) has an attachment-vertex
    on \(T_j\). If \(B\) is trivial,
    then it consists of one edge with
    an end in the component
    \(B_j - A_j\) of \(G - V(F_{j-1})\),
    and hence \(B \subseteq B_j\).
    If \(B\) is non-trivial, then
    it is obtained by adding attachment-vertices
    to a component of \(G - V(F_{j})\)
    adjacent to \(T_j\).
    That component is contained in
    \(B_j - A_j\), so again
    \(B \subseteq B_j\). 

    Now suppose that \(B \subseteq B_j\).
    If \(B\) is trivial, then its two
    attachment-vertices belong to
    \(A_j \cup V(T_j)\).
    Since \(A_j\) is an independent set in
    \(B_j\), at least one attachment-vertex of
    \(B\) is on \(T_j\).
    If \(B\) is non-trivial, then
    since \(B \subseteq B_j\), it
    is obtained from a component
    of \((B_j - A_j) - V(T_j)\) by
    adding all vertices adjacent to
    it as attachment-vertices.
    Since \(B_j - A_j\) is connected,
    at least one of these attachment-vertices will
    lie on \(T_j\).
\end{proof}

\begin{claim}\label{FiBridgeIsBj}
    For \(j \in \{1, \ldots, m\}\),
    every non-trivial \(F_j\)-bridge \(B\)
    is equal to \(B_k\) for some
    \(k \in \{j+1, \ldots, m\}\).
\end{claim}
\begin{proof}
    The vertex-sets of the
    trees \(T_1, \ldots, T_m\)
    partition \(V(G)\), so
    there exists the least
    \(k \in \{j+1, \ldots, m\}\)
    that contains a non-attachment-vertex of
    \(B\). Hence, \(B\) intersects
    \(F_{k-1}\) only in its attachment-vertices,
    and they all belong to \(F_j\),
    so \(B\) is an \(F_{k-1}\)-bridge
    that intersects \(T_k\), and
    therefore \(B = B_k\).
\end{proof}

Although we do not use this in our proof, we now show that \(\TT\) is a chordal partition.

\begin{claim}\label{ChordalPartition}
    \(\TT\) is a chordal partition with
    \(\tw(G/\TT) \le 2\).
\end{claim}
\begin{proof}
Clearly $\TT$ is a connected partition since each part has a spanning subtree.
    Let \(j \in \{1, \ldots, m\}\).
    If \(T_j\) is adjacent in \(G\)
    to a tree \(T_i\) with \(i < j\),
    then, since
    \(T_j \subseteq B_j - A_j\),
    the \(F_{j-1}\)-bridge \(B_j\) has
    an attachment-vertex on \(T_i\), that is,
    \(i \in X_j\).
    Since \(|X_j| \le 2\), the tree
    \(T_j\) can be adjacent to at most
    two of the trees
    \(T_1, \ldots, T_{j-1}\).
    It remains to show that
    if \(T_j\) is adjacent to
    two trees \(T_i\) and \(T_{i'}\)
    with \(i < i' < j\), then
    \(T_i\) is adjacent to \(T_{i'}\)
    in \(G\).
    Since \(T_j\) is adjacent to \(T_{i'}\),
    we have \(i' \in X_j\), and therefore,
    by \cref{FjBridgeInBj}, we have
    \(B_j \subseteq B_{i'}\).
    The attachment-vertices of \(B_j\) on \(T_i\)
    are thus attachment-vertices of \(B_{i'}\),
    so \(i \in X_{i'}\).
    By \cref{TiAdjacentToTj},
    \(T_i\) is adjacent to \(T_{i'}\).
\end{proof}

The following property of our chordal partition
will play a key role in the proof.
\begin{claim}\label{BBkPathAkout}
    Let \(j, k \in \{1, \ldots, m\}\) be such that \(B_k\) is an \(F_j\)-bridge
    contained in \(B_j\), let
    \(B\) be a (possibly trivial) \(F_j\)-bridge contained in \(B_j\)
    that is distinct from \(B_k\)
    and has an attachment-vertex in \(A_j\), and let \(Q\) be a \(V(B)\)--\(V(B_k)\) path in 
    \(B_j - V(T_j^0)\).
    Then the end of \(Q\) in \(V(B_k)\) belongs to \(A_k^{\out}\).  
\end{claim}

\begin{proof}
  By \cref{FjBridgeInBj},
  each of the \(F_j\)-bridges
  \(B\) and \(B_k\) has an
  attachment-vertex on \(T_j\).
  The \(F_j\)-bridge \(B\) has attachment-vertices
  on at most two components of \(F_j\),
  so the attachment-vertices of \(B\) in \(A_j\)
  must belong to one tree \(T_i\) with \(i < j\).
  We show that every attachement-vertex of $B_k$ that does not lie on $T_j$
  belongs to $T_i$.
  By \cref{NoAttachmentOnTj0}, the non-trivial \(F_j\)-bridge \(B_k\) is  disjoint from \(T_j^0\). Likewise, if \(B\) is non-trivial, then it is is disjoint from \(T_j^0\). Otherwise \(B\) is trivial and it can have attachments on \(T_j^0\). 
  Let \(B' := B - (V(B) \cap V(T_j^0))\). Hence, \(B'\) is a connected graph
  that contains all attachment-vertices of \(B\) on \(T_i\) and an end of \(Q\).
  The graph \(B_k \cup Q \cup B'\) is therefore
  a connected subgraph of \(B_j - V(T_j^0)\)
  that intersects \(T_i\).
  Therefore, by \cref{VTj0Separates},
  the graph
  \(B_k \cup Q \cup B'\) intersects \(A_j\) only in vertices belonging to \(T_i\),
  so indeed any attachment-vertices of \(B_k\) not on \(T_j\) must lie on \(T_i\).
  See \cref{FigAkout}.
  \begin{figure}
      \centering
      \includegraphics{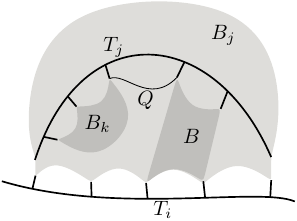}
      \caption{Illustration of \cref{BBkPathAkout}. The end of \(Q\) in \(B_k\) must belong to \(A_k^{\out}\)}
      \label{FigAkout}
  \end{figure}

  We claim that the only face of \(T_i \cup T_j \cup B_k\) whose boundary intersects
  both \(T_i\) and \(T_j\) is the outer-face.
  By our invariant, this is true when \(B_k\) has attachment-vertices in both \(T_i\) and \(T_j\),
  so suppose that \(B_k\) has attachment-vertices only on \(T_j\).
  Then \(T_i \cup T_j \cup B_k\) has two
  components \(T_i\) and \(T_j \cup B_k\).
  The graph \(T_j\) is in the closure of the outer-face of \(B_k\),
  and the graph \(T_i\) is in the closure of the outer-face of \(B_j\).
  Since \(T_j \cup B_k \subseteq B_j\), this means that \(T_i\) is on the outer-face of
  \(T_j \cup B_k\), and the only face of \(T_i \cup T_j \cup B_k\) whose
  boundary intersects \(T_i\) and \(T_j\) is the outer-face.
  Therefore, \(B\) is contained in the closure of the
  outer-face of \(T_i \cup T_j \cup B_k\).
  By \cref{NoAttachmentOnTj0}, \(B_j\) is disjoint from \(T_i^0\).
  Since \(Q \subseteq B_j - V(T_j^0)\), this means that
  \(Q\) is disjoint from \(T_i^0\) and \(T_j^0\), and thus
  \(Q\) can intersect the trees \(T_i\) and \(T_j\) only in their leaves.
  Furthermore, \(Q\) intersects \(B_k\) only in one end, so the path \(Q\)
   belongs to the closure of the outer-face of \(T_i \cup T_j \cup B_k\) together with
  \(B\). Hence, the path \(Q\) intersects \(B_k\) in a vertex on the boundary of the
  outer-face of \(T_i \cup T_j \cup B_k\). That vertex is an attachment-vertex of
  \(B_k\) in \(A_k^{\out}\).
\end{proof}

The following claim, while not used in the main proof, provides helpful intuition for the more complicated proof that follows. The proof of this claim does not rely on the value of \(\tau\), and even works with
\(\tau = 1\). Also, the trees \(T_j^0\)
do not need to minimise the number of edges for
this proof to work. These
properties will be useful later,
in the proof of \cref{BlockingPlanar}.

\begin{claim}
  The partition \(\mathcal{T}\) is 6-blocking.
\end{claim}
\begin{proof}
Consider a \(\mathcal{T}\)-clean path \(P\) in \(G\). 
We now show that the length of \(P\) is at most
\(6\). Let \(T_i\) be the tree that intersects \(P\) and has the
smallest \(i\).
Since \(P\) is \(\mathcal{T}\)-clean, it intersects
\(T_i\) in only one vertex, which splits \(P\) into
two edge-disjoint paths, each intersecting \(T_i\)
only in one of its ends.
Therefore, it suffices to show that if
\(Q = (x_0, \ldots, x_p)\) is a \(\mathcal{T}\)-clean
path such that \(V(Q) \cap V(F_i) = \{x_0\} \subseteq V(T_i)\), then
\(p \le 3\).

Suppose towards a contradiction that \(p \ge 4\).
Since \(V(Q) \cap V(F_i) = \{x_0\}\),
the path
\(Q\) is contained in a non-trivial \(F_i\)-bridge.
By \cref{FiBridgeIsBj},
that \(F_i\)-bridge is equal to \(B_j\)
for some \(j \in \{i+1, \ldots, m\}\).
Fix the largest \(j \in \{i+1, \ldots, m\}\)
such that \(Q \subseteq B_j\).
We split the argument into two cases based on whether
the path \(Q\) intersects \(T_j^0\) or
not.

Suppose first that \(x_\alpha \in V(T_j^0)\) for some
\(\alpha \in \{1, \ldots, p\}\).
Since \(Q\) is \(\mathcal{T}\)-clean, \(x_\alpha\)
is the only vertex of \(Q\) on \(T_j\).
In particular,
the vertex \(x_{\alpha-1}\)
is adjacent to \(T_j^0\)
in \(B_j\) and does not belong to
\(T_j\), so \(x_{\alpha-1} \in A_j\).
The path \(x_0 Q x_{\alpha-1}\) is disjoint from
\(V(T_j^0)\), so by
\cref{VTj0Separates}
it contains attachment-vertices of \(B_j\) on at most one
component of \(F_{j-1}\).
Since \(x_0 \in V(T_i)\) and \(x_{\alpha-1} \in A_j\),
this implies that \(x_{\alpha-1} \in V(T_i)\), and
therefore \(\alpha - 1 = 0\), that is \(\alpha = 1\).

The vertex \(x_0\) is the only vertex of \(Q\) on \(T_i\),
and the vertex \(x_1\) is the only vertex of \(Q\) on \(T_j\).
We have \(x_1 \in V(T_j^0)\), so by definition of \(T_j\)
the vertex \(x_2\) is an attachment-vertex of \(B_j\)
on a tree \(T_{i'}\) distinct from \(T_i\).
By our choice of \(i\), we have \(i < i' < j\).
Hence, \(B_j\) has attachment-vertices only on \(T_i\) and \(T_{i'}\).
Since \(Q \subseteq B_j\), and \(Q\) is \(\TT\)-clean,
this implies \(V(Q) \cap V(F_j) = \{x_0, x_1, x_2\}\).
Since \(p \ge 4\), the path \(x_2 Q x_p\) is contained in a
non-trivial \(F_j\)-bridge which,
by \cref{FiBridgeIsBj} is equal to \(B_k\)
for some \(k \in \{j+1, \ldots, m\}\).
Since \(Q \subseteq B_j\), we have \(B_k \subseteq B_j\).
The edge \(x_1 x_2\) is a trivial \(F_j\)-bridge
contained in \(B_k\)
that contains an attachment-vertex in \(A_j\), and
its attachment-vertex \(x_2\) belongs to \(B_k\)
(see \cref{6BlockingA}).
Hence, by \cref{BBkPathAkout} applied to
the trivial path consisting of the vertex
\(x_2\) alone, we have \(x_2 \in A_k^{\out}\).

We have \(B_k \subseteq B_j\),
so by \cref{FjBridgeInBj},
the \(F_j\)-bridge \(B_k\) has an attachment-vertex on \(T_j\).
The vertex \(x_2\) is an attachment-vertex of \(B_k\)
on \(T_{i'}\), so
\(B_k\) has attachment-vertices only on \(T_{i'}\)
and \(T_j\).
Since \(V(Q) \cap V(T_{i'}) = \{x_2\}\)
and \(V(Q) \cap V(T_j) = \{x_1\}\),
we have \(x_3 Q x_p \subseteq B_k - A_k\).
Since \(x_2 \in A_k^{\out}\) this implies that
\(x_3 \in D_k \subseteq V(T_k^0)\), and thus
\(x_4 \in V(T_k)\). Hence \(\{x_3, x_4\} \subseteq V(T_k)\), contrary to the assumption that
\(Q\) is \(\TT\)-clean.

Now consider the case when
\(Q\) is disjoint from \(T_j^0\).
We have \(x_0 \in V(T_i)\) and \(x_1 \not \in V(T_i)\),
so \(x_0 x_1 \not \in E(F_j)\). Let \(B\) be the
\(F_j\)-bridge containing the edge
\(x_0 x_1\).
Because \(x_0 x_1 \in E(B_j)\), we have \(B \subseteq B_j\),
so by \cref{FjBridgeInBj},
\(B\) has an attachment-vertex on \(T_j\). 
The vertex \(x_0\) is an attachment-vertex of \(B\) on \(T_i\),
so \(B\) has attachment-vertices only on \(T_i\) and \(T_j\). 
Observe that \(Q \not \subseteq B\); indeed,
if \(B\) is trivial, this follows from the fact that
\(p \ge 4\), and if \(B\) is a non-trivial
\(F_j\)-bridge, then by
\cref{FiBridgeIsBj}, we have \(B = B_k\) for some
\(k \in \{j+1, \ldots, m\}\), and
\(Q \not \subseteq B_k\) by our choice of \(j\).
Since \(Q \not \subseteq B\) and \(x_0 x_1 \in E(B)\),
\(Q\)  contains a vertex \(x_\alpha\) that is
an attachment-vertex of \(B\) distinct from \(x_0\).
Since \(B\) has attachment-vertices only on \(T_i\) and
\(T_j\) and \(Q\) is \(\TT\)-clean, the vertex
\(x_{\alpha}\) is the only vertex of \(Q\)
on \(T_j\).

We claim that \(\alpha \le 2\).
If \(B\) is trivial, then \(\alpha = 1 \le 2\), so suppose
that \(B\) is non-trivial, and thus \(B = B_k\)
for some \(k \in \{j+1, \ldots, m\}\).
We have \(x_0 \in V(T_i)\) and \(x_\alpha \in V(T_j)\),
so by \cref{VTj0Separates},
the path \(x_0 Q x_\alpha\) must intersect
\(T_k^0\) in a vertex \(x_{\alpha'}\)
with \(0 < \alpha' < \alpha\). Since \(Q\) is
\(\TT\)-clean, the vertex \(x_{\alpha'}\)
is the only vertex of \(Q\) in \(V(T_k)\).
Hence, by definition of
\(T_k\), the vertices
\(x_{\alpha'-1}\) and \(x_{\alpha'+1}\) are
attachment-vertices of \(B_k\), and therefore belong to
\(V(T_i) \cup V(T_j)\).
The only vertex of \(Q\) in \(V(T_i)\) is \(x_0\), and
the only vertex of \(Q\) in \(V(T_j)\) is \(x_\alpha\),
so \(x_{\alpha'-1} = x_0\) and \(x_{\alpha'+1} = x_\alpha\).
Hence \(\alpha' = 1\) and \(\alpha = 2\). 
This proves \(\alpha \le 2\).

Since \(Q \subseteq B_j - V(T_j^0)\),
\cref{VTj0Separates} implies that
the only component of \(F_{j-1}\) intersected by
\(Q\) is \(T_i\).
Hence, \(V(Q) \cap V(F_j) = \{x_0, x_\alpha\}\).
Since \(\alpha \le 2\) and \(p \ge 4\),
the path \(x_\alpha Q x_p\) is contained in a
non-trivial \(F_j\)-bridge, which equals \(B_k\)
for some \(k \in \{j+1, \ldots, m\}\).
See \cref{6BlockingB}.
The \(F_j\)-bridge \(B\) is contained in \(B_j\) and
has an attachment-vertex in \(A_j\),
and the vertex \(x_\alpha\) is an attachment-vertex of
\(B_{k}\) in \(V(T_j)\).
Hence, by \cref{BBkPathAkout} applied to the trivial
path consisting of the vertex \(x_\alpha\) alone,
we have \(x_{\alpha} \in A_{k}^{\out}\).
By \cref{NoAttachmentOnTj0}, 
\(B_{k}\) is disjoint from \(T_j^0\), and it
is contained in the same component of \(B_j - V(T_j^0)\) as \(Q\). Hence,
by \cref{VTj0Separates}, \(B_{k}\) can only have
attachment-vertices in \(V(T_i)\) and \(V(T_j)\).
Since \(Q\) is \(\TT\)-clean with \(x_0 \in V(T_i)\)
and \(x_\alpha \in V(T_j)\),
we have \(x_{\alpha+1} Q x_p \subseteq B_{k} - A_{k}\).
Since \(x_\alpha \in A_{k}^{\out}\), this implies
\(x_{\alpha+1} \in D_{k} \subseteq V(T_{k}^0)\), and thus
\(x_{\alpha+2} \in V(T_{k})\).
Therefore,
\(\{x_{\alpha+1}, x_{\alpha+2}\} \subseteq V(T_{k})\),
contrary to the assumption that \(Q\) is \(\TT\)-clean.
\begin{figure}
     \centering
     \begin{subfigure}[b]{0.49\textwidth}
         \centering
         \includegraphics{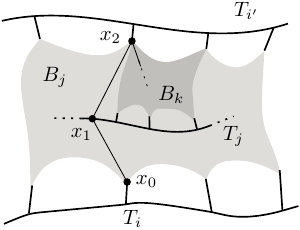}
         \caption{$V(Q) \cap V(T_j^0) \neq \emptyset$}
         \label{6BlockingA}
     \end{subfigure}
     \hfill
     \begin{subfigure}[b]{0.49\textwidth}
         \centering
         \includegraphics{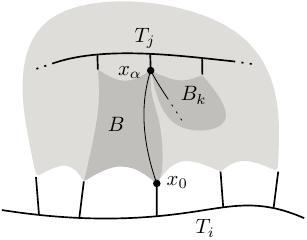}
         \caption{$V(Q) \cap V(T_j^0) = \emptyset$}
         \label{6BlockingB}
     \end{subfigure}
        \caption{The two cases in the proof that the length of \(Q\)
        is at most \(3\).}
\end{figure}
\end{proof}

Although the partition \(\TT\) is \(6\)-blocking, its parts can be arbitrarily large. The next
step of our construction refines the chordal partition.

\section{Refinement of the Chordal Partition}
\label{Refinement}

In order to define our refinement
of the chordal
partition \(\TT\), we need to study its properties
in more detail.

\begin{claim}\label{DjSize}
  For each \(j \in \{1, \ldots, m\}\), \(|D_j| < \Delta^{40}\).
\end{claim}
\begin{proof}
    If \(B_j\) has no attachment-vertices, then
    \(|D_j| = 1 \le \Delta^{40}\), so suppose that
    \(B_j\) has some attachment-vertices.
    \(B_j\) has attachment-vertices in at most two components of \(F_{j-1}\), and
    on each of them \(A_j^{\out}\)
    has one or two vertices, so \(|A_j^{\out}| \le 4\).
    Since each vertex in \(D_j\) is at distance at most
    \(\tau\) from a vertex in \(A_j^{\out}\), 
    \begin{equation*}
    |D_j| \le
    |A_j^{\out}|(\Delta^0 + \ldots + \Delta^{\tau}) <
    4\Delta^{\tau+1} < 
    \Delta^{\tau+3} = 
    \Delta^{40}.\qedhere
    \end{equation*}
\end{proof}

Two paths in a graph are \defn{internally disjoint} if none of them contains an inner vertex of another, and a path is
\defn{internally disjoint} from a set of vertices \(D\)
if no inner vertex of the path belongs to \(D\). 

\begin{claim}
\label{Steiner}
For each \(j \in \{1, \ldots, m\}\), the tree \(T_j^0\) is the union of a family \(\PP_j\) of at most \(2\Delta^{40}\) geodesic paths  which are pairwise internally disjoint and internally disjoint from $D_j$.
\end{claim}

\begin{proof}
Let \(S\) denote the set of all
vertices of \(T_j^0\)
with degree at least \(3\).
Since \(T_j^0\) is the tree in \(B_j - A_j\)
which contains all vertices in \(D_j\) and has
the smallest possible number of edges, 
every leaf of \(T_j^0\) belongs to \(D_j\),
so \(T_j^0\) has at most \(|D_j|\) leaves, and 
thus  \(|S| \le |D_j|\).
The tree \(T_j^0\) is a subdivision of a tree
with vertex-set \(S \cup D_j\).
Therefore, \(T_j^0\) is the union of a set
\(\PP\) of at most \(|S \cup D_j|\)
pairwise internally
disjoint paths such that each \(P \in \PP\)
has its ends in \(S \cup D_j\) and
is internally disjoint from \(S \cup D_j\).
We have \(|\PP| \le |S \cup D_j| \le 2|D_j|\),
so by \cref{DjSize}, \(|\PP| \le 2\Delta^{40}\). Suppose towards a contradiction that some path \(P \in \PP\) is not
geodesic in \(B_j - A_j\), and let \(P'\) be a geodesic path in
\(B_j - A_j\) between the
ends of \(P\). Hence, \(P'\) has less edges than \(P\), so any spanning tree of \(P'\cup \bigcup_{Q \in \PP \setminus \{P\}}Q\) has less edges than \(T\) and contains all vertices in \(D_j\), which is a contradiction. 
\end{proof}

An important property of geodesic paths is that
the distances between vertices are preserved in
them. We show that the tree \(T_j^0\)
`approximates' the distances between
its vertices in \(B_j - A_j\).

\begin{claim}\label{ApproxGeodesicTj0}
  Let \(j \in \{1, \ldots, m\}\),
  and let \(x, y \in V(T_j^0)\).
  Then
  \[
    \dist_{T_j^0}(x, y) <
      \Delta^{40} \dist_{B_j - A_j}(x, y).
  \]
\end{claim}
\begin{proof}
  Let \(P\) be a
  geodesic
  \(x\)--\(y\) path in \(B_j - A_j\).
  We have
  \(\dist_P(x, y) = \dist_{B_j - A_j}(x, y)\),
  so we need to show that
  \(\dist_{T_j^0}(x, y) < \Delta^{40}\dist_P(x, y)\).

First consider the case when \(P\)
is internally disjoint from \(V(T_j^0)\).

  Let \(z_0, \ldots, z_s\) denote the sequence of all vertices of the path
  \(xT_j^0y\) that belong to \(D_j \cup \{x, y\}\)
  or have degree at least \(3\) in \(T\),
  ordered by increasing distance from \(x\) (so that \(z_0 = x\) and \(z_s = y\)).

  We show that \(s < |D_j|\) by associating a distinct vertex
  \(z_i' \in D_j\) to each \(z_i\). Let $i\in \{0,\dots,s\}$.
  If \(z_i \in D_j\), then let \(z_i' := z_i\).
  Otherwise $z_i \not \in D_j$, and either $z_i$ is an end of \(x T_j^0 y\) but not a leaf of \(T_j^0\), or \(z_i\) has degree at least $3$ in \(T_j^0\).
  In both cases, there exists a leaf $z_i'$ in $T_j^0$ such that $z_i$ is adjacent to the component of \(T_j^0 - V(xT_j^0y)\) that contains $z_i'$. By our choice of $T_j^0$, we have $z_i'\in D_j$.
  Clearly, the vertices \(z_0', \ldots, z_s'\)
  are distinct, so \(s < |D_j|\).

  For each \(i \in \{0, \ldots, s-1\}\), let \(T(i)\)
  denote the graph obtained from \(T_j^0\) by removing all edges and inner vertices of \(z_i T_j^0 z_{i+1}\) and adding the path \(P\).
  The path \(P\) has ends in \(z_0\) and \(z_s\), and is
  internally disjoint from
  \(V(T_j^0)\), so
  \(T(i)\) is a tree.
  This tree still contains \(D\), so by definition of \(T_j^0\), 
  \[
    |E(T_j^0)| \le |E(T(i))| = |E(T_j^0)| - |E(z_iT_j^0z_{i+1})|+|E(P)|,
  \]
  so \(\dist_{T_{j}^0}(z_i, z_{i+1}) = |E(z_iT_j^0z_{i+1})|\le |E(P)|\). Therefore,
  \begin{align*}
      \dist_{T_{j}^0}(x, y)
      =\sum_{i=0}^{s-1}|E(z_{i}T_j^0z_{i+1})|
      \le s\cdot|E(P)|
      < |D_j|\cdot\dist_P(x, y)
      \le \Delta^{40}\cdot\dist_P(x, y),
  \end{align*}
  where the last inequality follows from
  \cref{DjSize}.

  It remains to consider the case when
  \(P\) has at least one
  inner vertex in \(V(T_j^0)\).
  Let \(w_0, \ldots, w_n\) denote the vertices
  in \(V(P) \cap V(T_j^0)\) ordered by increasing distance from
  \(x\) in \(P\) (so that \(w_0 = x\) and \(w_n = y\)).
  For each \(i \in \{0, \ldots, n-1\}\), the path \(w_i P w_{i+1}\)
  has no inner vertices in \(V(T_j^0)\), so \(\dist_{T_j^0}(w_i, w_{i+1}) <
  \Delta^{40}\cdot\dist_{P}(w_i, w_{i+1})\), and thus
  \[
  \dist_{T_j^0}(x, y) \le
  \sum_{i=0}^{n-1}\dist_{T_j^0}(w_i, w_{i+1})
  < \sum_{i=0}^{n-1}\Delta^{40}\cdot\dist_P(w_i, w_{i+1}) =
  \Delta^{40}\cdot\dist_P(x, y). \qedhere
  \]
\end{proof}

Let $c\in \NN$.
For any vertex
\(x \in V(G)\), the number of
vertices \(x' \in V(G)\) with
\(\dist_G(x, x') \le c\) is at most
\(\sum_{i=0}^c\Delta^i\), and therefore
less than \(\Delta^{c+1}\).
Therefore, for any edge \(e \in E(G)\),
the number of edges \(e' \in E(G)\) with
\(\dist_G(e, e') \le c\) is less than
\(2\Delta^{c+2}\)
(since any such \(e'\) is incident to a vertex
at distance at most \(c\) from one of the
two endpoints of \(e\)).
We use these bounds implicitly
in the following part of this section.

In a graph \(J\),
we say that a set of edges \(M \subseteq E(J)\) is \defn{\(d\)-independent} if for any pair of distinct edges
\(e_1, e_2 \in M\) we have
\(\dist_J(e_1, e_2) > d\).
We aim to refine the partition \(\TT\)
be removing a set of edges \(M_j \subseteq E(T_j^0)\)
from each \(T_j\), and letting the
components of
the resulting forests be the parts of
the partition.
The precise description of the desired properties
of the sets \(M_j\)
will be given in \cref{SetsMi}.
Roughly speaking, we want the edges in each \(M_j\) to be far away from each other, from other sets \(M_{j'}\), and from the set \(D_j\), while ensuring that the components of \(T_j - M_j\) have bounded size. In order to formalise being far away, we need the following definition. Let \(i \in \{1, \ldots, m\}\),
and suppose that the set \(M_i \subseteq E(T_i^0)\) is already defined.
Let \(S\) be a set of vertices or a set of edges in \(B_i - V(T_i^0)\).
The \defn{mixed distance} of \(S\) from \(M_i\) is
\[
\mathdefn{\mdist_i(S)} := \min \{
\dist_{B_i - A_i}(M_i, v) + \dist_{B_i - V(T_i^0)}(v, S):
v \in V(T_i) \setminus V(T_i^0)\}.
\]
Our goal is to construct the sets \(M_j\) so that for an appropriate constant \(c\) (specified in the next section), for each \(j \in \{1, \ldots, m\}\) with \(X_j \neq \emptyset\), we have
\(\mdist_i(M_j) > c\) for all \(i \in X_j\).

The sets \(M_j\) will be constructed
one-by-one, where each set \(M_j\) is obtained from \(T_j^0\) by selecting an appropriate set of edges from each geodesic path in $\PP_j$, using the following claim, which exploits the fact that $G$ is a plane graph.

\begin{claim}\label{DistanceDIndependentSet}
    Let \(c \ge 1\), 
    let \(d := (8c+12)\Delta^{c+2}\), and
    let \(n_0 \ge d + 2c\).
    Let \(P\) be a geodesic path in \(B_j - A_j\)
    for some \(j \in \{1, \ldots, m\}\) with $X_j\neq \emptyset$,
    and suppose that for
    each \(i \in X_j\) we are given a set \(M_i \subseteq E(T_i^0)\)
    that is
    \((d+2c)\)-independent in \(B_i - A_i\),
    Then there exists a set \(M_P \subseteq E(P)\)
    that is \((d+2c)\)-independent in \(B_j - A_j\)
    such that each component of \(P - M_P\)
    has length at least \(\min\{n_0, |E(P)|\}\)
    and less than \(5n_0\), and for each \(i \in X_j\)
    we have \(\mdist_i(M_P) > c\).
\end{claim}
\begin{proof}
  We may assume that the length of \(P\) is at least
  \(5n_0\), as otherwise
  the lemma is satisfied by \(M_P = \emptyset\).
  Let \(x\) and \(y\) denote the ends of \(P\).
  Let \(\{P_1, \ldots, P_t\}\) be an inclusion-maximal family of pairwise
  vertex-disjoint subpaths of \(P\) each of length
  \(d\) such that 
  \(\dist_P(V(P_\alpha), V(P_\beta)) \ge n_0\)
  for distinct \(\alpha,\beta \in \{1, \ldots, t\}\),
  and \(\dist_P(V(P_\alpha), \{x, y\}) \ge n_0\) for
  every \(\alpha \in \{1, \ldots, t\}\).
  Since \(n_0 > d\) and
  the length of \(P\) is at least \(5n_0\),
  we have \(t > 0\).
  Consider any maximal subpath \(P'\subseteq P\)
  internally disjoint from each
  of the paths \(P_1,\ldots,P_t\).
  Then the length of \(P'\) is at least \(n_0\).
  Since our family of paths is inclusion-maximal,
  the length of \(P'\) is less than
  \(d + 2n_0\) as otherwise we would
  be able to extend our family with a path of
  length \(d\) obtained from \(P'\) by removing
  at least \(n_0\) vertices from each side.
  Since \(d  + 2n_0 < 3n_0\), we conclude that
  the length of any
  such \(P'\) is at least \(n_0\) and less than
  \(3n_0\).

  We claim that each path \(P_\alpha\) contains an
  edge \(e_\alpha\) such that for every \(i \in X_j\)
  we have \(\mdist_i(e_\alpha) > c\).
  Since \(|X_j| \le 2\), it suffices to show that
  for each \(i \in X_j\), there is less than
  \(d/2\) edges \(e \in E(P)\) with
  \(\mdist_i(e) \le c\).
  Fix \(\alpha \in \{1, \ldots, t\}\) and \(i \in X_j\).
  Partition \(M_i\) into two sets \(M_i'\) and \(M_i''\)
  by assigning each edge \(e' \in M_i\) to \(M_i'\)
  if \(\dist_{B_i - A_i}(e', V(P_\alpha)) \le c\), and to
  \(M_i''\) if \(\dist_{B_i - A_i}(e', V(P_\alpha)) > c\).
  Since
  \(M_i\) is \((d+2c)\)-independent in
  \(B_i - A_i\), and the length of
  \(P_\alpha\) is at most \(d\), the set
  \(M_i'\) contains at most one edge.

  For every \(e' \in M_i''\), let \(U_{e'}\)
  denote a subtree of \(G\) on all vertices at distance at
  most \(c\) from \(e'\) in \(B_i - A_i\) such
  that each \(u \in V(U_{e'})\) has the same
  distance from \(e'\) in \(U_{e'}\) as in
  \(B_i - A_i\) (one can think of \(U_{e'}\) as
  a ``\textsc{bfs}-spanning tree rooted at the edge \(e'\)'') 
  Since the set \(M_i\) is \((d+2c)\)-independent
  in \(B_i - A_i\), the trees \(U_{e'}\) are
  pairwise vertex-disjoint, and by definition
  of \(M_i''\) the trees \(U_{e'}\) are disjoint
  from the path \(P_\alpha\).
  For each \(e' \in M_i''\), let
  \(Z_{e'} := N_{B_i}(V(U_{e'})) \cap A_i\).
  For each \(z \in Z_{e'}\), define a
  \(V(T_i)\)--\(A_i\) path
  \(Q(e', z)\) as follows. 
  Let
  \(y\) be a vertex of \(U_{e'}\) adjacent to \(z\) in \(B_i\) which minimises $\dist_{U_{e'}}(e',y)$ (and thus also minimises $\dist_{B_i - A_i}(e',y)$). Let \(x\) be the vertex on the path between \(y\) and \(e'\) which lies on \(V(T_i)\)  and minimises $\dist_{U_{e'}}(x,y)$ (this is well defined since  \(e'\) has both ends in \(V(T_i)\)).
  Then the path \(Q(e', z)\) is obtained
  from \(x U_{e'} y\) by adding the vertex \(z\)
  attached to \(y\).
  Each pair of distinct paths
  \(Q_1 = Q(e_1', z_1)\)
  and \(Q_2 = Q(e_2', z_2)\)
  is \defn{consistent},
  meaning that if 
  their intersection \(Q_1 \cap Q_2\) is
  not empty, then \(Q_1 \cap Q_2\) is a 
  path with an end in a common end of \(Q_1\) 
  and \(Q_2\). 

  Let \(J\) denote the union of \(T_i\)
  and the paths
  \(Q(e', z)\) for all \(e' \in M_i''\) and
  \(z \in Z_{e'}\).
  The graph \(J\) is a subgraph of \(B_i\)
  that intersects \(A_i\) only in the sets \(Z_{e'}\).
  Thus, all vertices in the sets \(Z_{e'}\)
  are incident with the outer face of \(B_i\).
  Let \(J^+\) denote the planar graph obtained from
  \(J\) by adding a new vertex \(w_+\)
  on the outer-face of \(B_i\) and making it adjacent to all vertices in the sets
  \(Z_{e'}\).
  See \cref{MjClaim}.
  \begin{figure}
      \centering
      \includegraphics{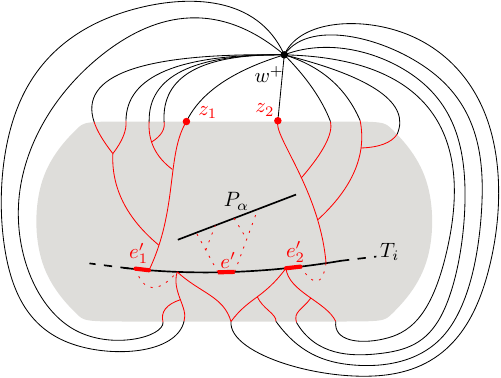}
      \caption{Illustration of a possible scenario in \cref{DistanceDIndependentSet}. The set \(M_i'\) consists of the
      edge \(e'\), and the boundary of the face \(f\) contains two paths
      \(Q(e'_1, z_1)\) and \(Q(e'_2, z_2)\). Hence,
      \(S = \{e'\} \cup E(Q(e'_1, z_1)) \cup E(Q(e'_2, z_2))\).}
      \label{MjClaim}
  \end{figure}
  Since \(P_\alpha \subseteq B_i - A_i\) and
  \(\dist_{B_i - A_i}(V(P_\alpha), M_i'') > c\), the path
  \(P_\alpha\) is disjoint from \(J^+\) and therefore
  the path \(P_\alpha\) is contained in a face \(f\) of \(J^+\). 
  Let \(S\) denote the set of all edges in
  \(E(J) \setminus E(T_i)\) on the boundary of \(f\).
  Since the paths \(Q(e', z)\) are pairwise consistent
  \(V(T_i)\)--\(A_i\) paths,
  the edges in \(S\) can be covered by 
  the union of at most two paths of the form
  \(Q(e', z)\).
  In particular, \(|S| \le 2c+2\).

  We claim that for every \(e \in E(P_\alpha)\)
  with \(\mdist_i(e) \le c\) we have
  \(\dist_{B_i}(S \cup M_i', e) \le c\).
  Suppose that \(\mdist_i(e) \le c\).
  Hence, there exist \(e' \in M_i\) and \(v \in V(T_i) \setminus V(T_i^0)\)
  that satisfy
  \(\dist_{B_i - A_i}(e', v) + \dist_{B_i - V(T_i^0)}(v, e) \le c\).
  If \(e' \in M_i'\), then we indeed have
  \(\dist_{B_i}(S \cup M_i', e) \le c\), so we
  assume that \(e' \in M_i''\).
  Let \(R\) be a shortest path between \(v\) and \(e\) in \(B_i - V(T_i^0)\).
  Since \(\dist_{B_i - A_i}(e, e') > c\), the path \(R\) must intersect
  \(A_i\). Let \(z\) be the
  vertex of \(R\) that belongs to \(A_i\) and is closest to \(v\) on \(R\).
  Hence, \(z \in Z_{e'}\).
  Since \(J^+\) contains \(Q(e', z)\),
  the path \(Q(e', z)\) is disjoint from the interior of \(f\).
  Therefore, the subpath of \(R\) between \(z\) and \(e\) must intersect
  the boundary of \(f\) in a vertex \(u\),
  and since \(R\) is disjoint from \(T_i^0\),
  the vertex \(u\) is an end of an edge in \(S\).
  Therefore, \(\dist_{B_i}(S, e) \le c\).
  This completes the proof that for every \(e \in E(P_\alpha)\),
  if \(\mdist_i(e) \le c\), then \(\dist_{B_i}(S \cup M_i', e) \le c\).
  Since \(|S \cup M_i'| \le |S| + |M_i'| \le 2c+3\),
  for each \(i \in X_j\) there exist less than \((4c+6)\Delta^{c+2}\)
  edges \(e \in E(P_\alpha)\) with \(\mdist_i(e) \le c\).
  Since \(|X_j| \le 2\) and the length of \(P_\alpha\) is
  \((8c+12)\Delta^{c+2}\),
  for each \(\alpha \in \{1, \ldots, t\}\) there exists an edge
  \(e_\alpha \in E(P_\alpha)\) such that
  \(\mdist_i(e_\alpha) > c\) for all \(i \in X_j\).
  
  Let \(M_P := \{e_1, \ldots, e_t\}\).
  Thus, \(\mdist_i(M_P) > c\)
  for each \(i \in X_j\).
  Since the distance between any two of the
  paths \(P_1, \ldots, P_t\) is at least
  \(n_0\) on \(P\), the set \(M_P\)
  is \(n_0\)-independent in \(P\).
  Because \(P\) is geodesic in \(B_j - A_j\)
  and \(n_0 \ge d + 2c\), the set \(M_P\) is
  \((d+2c)\)-independent in \(B_j - M_j\).

  It remains to show that the components
  of \(P - M_P\) have appropriate sizes.
  Let \(Q\) be a component of \(P - M_P\).
  Since \(M_P\) contains one edge from each of the subpaths \(P_1, \ldots, P_t\),
  the path
  \(Q\) intersects at most two of
  the paths \(P_1, \ldots, P_t\),
  and the total number of edges of \(Q\) shared with \(P_1, \ldots, P_t\) is at most \(2d\), and thus less than \(2n_0\).
  The edges of \(Q\) that do not belong to any of
  the paths \(P_1, \ldots, P_t\)
  induce a maximal subpath  of \(P\)
  internally disjoint from each of the
  paths \(P_1, \ldots, P_t\),
  which thus has length at least \(n_0\)
  and less than \(3n_0\).
  Hence, the length of \(P'\) is at least
  \(n_0\) and less than \(5n_0\).
\end{proof}

Finally, we are ready to construct the sets \(M_j\).

\begin{claim}\label{SetsMi}
Let \(c \ge 1\) and \(d := (8c+12)\Delta^{c+2}\).
There exists a family \(\{M_j \subseteq E(T_j^0) : j \in \{1, \ldots, m\}\}\) such that for every \(j \in \{1, \ldots, m\}\):
  \begin{enumerate}[(\alph*)]
    \item\label{DistanceDIndependent}
    $M_j$ is $(d + 2c)$-independent in
    \(B_j - A_j\),
    \item\label{DistanceDiMi} \(\dist_{T_j^0}(D_j, M_j) \ge 2\Delta^{40}\),
    \item\label{DistanceMiMj} 
    for each \(i \in X_j\), \(\mdist_i(M_j) > c\),
    \item\label{ComponentSize} each component of \(T_j^0 - M_j\) has 
    at most
    \(10\Delta^{80}(d+2c)\) vertices, and
    \item\label{LocalGeodesicity} for any pair of vertices \(x, y \in V(T_j^0)\) satisfying
    \(\dist_{B_j - A_{j}}(x, y) \le d+2c\) and \(E(x T_j^0 y) \cap M_j \neq \emptyset\), we have
    \(\dist_{T_j^0}(x, y) = \dist_{B_j - A_j}(x, y)\).
  \end{enumerate}
\end{claim}

\begin{proof}
  Let \(n_0 := \Delta^{40}(d+2c)\). 
  
  We construct the sets \(M_j\) by induction on
  \(j\).
  Let \(j \in \{1, \ldots, m\}\), and
  suppose that the sets \(M_i\) with \(i < j\) have 
  already been constructed. In particular,
  each \(M_i\) is $(d+2c)$-independent in
  \(B_i - A_i\).
  We now construct \(M_j\).
  
By \cref{Steiner} there is a family \(\PP_j\) of at most \(2\Delta^{40}\) pairwise
internally disjoint geodesic paths in
\(B_j - A_j\) whose union is \(T_j^0\),
and which are internally disjoint from
\(D_j\).
Observe that every inner vertex of a path \(P \in \PP_j\) has degree two in \(T_j^0\).
  
  For each \(P \in \PP_j\), let \(M_P \subseteq E(P)\) be the subset of edges obtained by applying \cref{DistanceDIndependentSet} to \(c\), \(n_0\) and \(P\). Thus, \(M_P\) is \((d+2c)\)-independent in
  \(B_j - A_j\), 
  \(\mdist_i(M_P) > c\) for each
  \(i \in X_j\), and each component of \(P - M_P\) is a path of length
  at least \(\min\{n_0, |E(P)|\}\) and
  less than \(5n_0\).

  We show that the set
  \(M_j := \bigcup\{M_P : P \in \PP_j\}\) satisfies the claim. For the proof of \cref{DistanceDIndependent},
  we need to show that \(M_j\) is
  \((d+2c)\)-independent in \(B_j - A_j\).
  Suppose towards a contradiction that
  there are distinct \(e_1, e_2 \in M_j\) with
  \(\dist_{B_j - A_j}(e_1, e_2) \le d + 2c\).
  By \cref{ApproxGeodesicTj0}, 
  \[
    \dist_{T_j^0}(e_1, e_2) <
    \Delta^{40}\dist_{B_j - A_j}(e_1,e_2) \le 
    \Delta^{40}(d+2c) = n_0.
  \] 
  However, if \(P \in \PP_j\) is the path containing \(e_1\), then
  the shortest path between \(e_1\) and \(e_2\) in \(T_j^0\) contains a component
  of \(P - M_P\), and therefore has length at least \(\min\{n_0, |E(P)|\}=n_0\) (since $e_1\in E(P)$); that is, \(\dist_{T_j^0}(e_1, e_2) \ge n_0\), a contradiction.

  For any \(P \in \PP_j\) and \(e \in M_P\),
  the distance between \(e\) and the ends of \(P\)
  is at least \(n_0 = \Delta^{40}(d+2c)\), and therefore at least \(2\Delta^{40}\).
  Since the paths in \(\PP_j\) are
  pairwise internally disjoint, and
  internally disjoint from \(D_j\),
  this implies that \(\dist_{T_j^0}(D_j, M_j) \ge 2\Delta^{40}\). Therefore \cref{DistanceDiMi} is satisfied.

  By definition of the sets \(M_P\),
  for any \(P \in \PP_j\) and
  \(i \in X_j\) we have
  \(\mdist_i(M_P) > c\),
  and therefore for each \(i \in X_j\)
  we have \(\mdist_i(M_j) > c\).
  This proves \cref{DistanceMiMj}.

  For \cref{ComponentSize}, observe that if a component
  \(T'\) of \(T_j^0 - M_j\) intersects a path \(P \in \PP_j\),
  then \(T' \cap P\) is a component of \(P - M_P\), so it
  has less than \(5n_0\) edges,
  and therefore at most \(5n_0\) vertices.
  Hence,
  \begin{align*}
    |V(T')| &\le |\PP_j|\cdot5n_0
      \le 2\Delta^{40}\cdot5\Delta^{40}(d+2c)
      = 10\Delta^{80}(d+2c).
  \end{align*}

  Finally, for the proof of \cref{LocalGeodesicity},
  let \(x,y \in V(T_j^0)\) be vertices satisfying
  \(\dist_{B_j - A_j}(x, y) \le d+2c\) and \(E(x T_j^0 y) \cap M_j \neq \emptyset\). Let \(e \in E(x T_j^0 y) \cap M_j\),
  and let \(P \in \PP_j\) be the path containing \(e\).
  By \cref{ApproxGeodesicTj0}, 
  \[
    \dist_{T_j^0}(x, y) <
    \Delta^{40}\cdot\dist_{B_j - A_j}(x,y) \le 
    \Delta^{40}\cdot(d+2c)=
    n_0.
  \]
  Since \(M_P \neq \emptyset\),
  every component of \(P - M_P\) has length at least \(n_0\),
  so the path \(x T_j^0 y\) does not contain a component of \(P - M_P\).
  Since \(E(x T_j^0 y) \cap M_P \neq \emptyset\)
  and all inner vertices of \(P\) have degree
  two in \(T_j^0\), this implies that \(x T_j^0 y\)
  is a subpath of \(P\),
  and since \(P\) is geodesic in \(B_j - A_j\), we have
  \(\dist_{T_j^0}(x, y) = \dist_{B_j - A_j}(x, y)\).
\end{proof}

\section{Analysis of the Partition}
\label{Analysis}

Let \(\ell := 222\), and let \(c := 2\ell+6=450\).
Fix a family
\(\{M_j : j \in \{1, \ldots, m\}\}\)
satisfying \cref{SetsMi} for our value of \(c\).
Let \(\RR\) denote the partition of \(V(G)\)
where each part is the vertex-set of a
component of \(\bigcup_{j=1}^m (T_j - M_j)\). By \cref{SetsMi}\ref{ComponentSize},
for each
\(j \in \{1, \ldots, m\}\), the size of every component of \(T_j^0 - M_j\) is at most
\(10\Delta^{80}((8c+12)\Delta^{c+2} + 2c)=
10\Delta^{80}(3612\,\Delta^{452}+900)\).
Since each component of \(T_j - M_j\)
can be obtained from a component of \(T_j^0 - M_j\) by attaching at most \(\Delta\) vertices
to each vertex of the component,
the width of \(\RR\) is at most
\((\Delta+1)\cdot10\Delta^{80}(3612\,\Delta^{452}+900)\).

To complete the proof of \cref{BlockingPlanar},
we show that \(\RR\) is \(\ell\)-blocking; 
that is, no
\(\RR\)-clean path in \(G\) has length greater than \(\ell\).
Since a subpath of an \(\RR\)-clean path is \(\RR\)-clean,
it suffices to show that there is no \(\RR\)-clean path of
length exactly \(\ell+1\), so in our analysis we focus
only on paths of length at most \(\ell+1\).

We start by proving some
properties of \(\RR\)-clean paths.

\begin{claim}\label{Distance}
  Let \(j \in \{1, \ldots, m\}\), and let
  \(Q = (x_0,\ldots,x_q)\) be an \(\RR\)-clean
  path in \(B_j - A_j\) with \(\{x_0, x_q\} \subseteq V(T_j)\) and \(q \in \{1, \ldots, \ell+1\}\).
  Then \(E(x_0 T_j x_q) \cap M_j \neq \emptyset\), and
  for each \(e \in E(x_0 T_j x_q) \cap M_j\) we have
  \(\dist_{T_j}(e, x_\alpha) \le q + 2\) for \(\alpha \in \{0, q\}\).
  In particular,
  \[
  \dist_{T_j}(M_j, x_\alpha) \le q + 2\quad
  \textrm{for \(\alpha \in \{0, q\}\)}.
  \]
\end{claim}
\begin{proof}
  For each \(\alpha \in \{0, q\}\), let
  \(x_\alpha'\) denote the vertex \(x_\alpha\) if
  \(x_\alpha \in V(T_j^0)\), or the vertex in
  \(V(T_j^0)\) that is adjacent to \(x_\alpha\)
  in \(T_j\) if \(x_\alpha \not \in V(T_j^0)\).
  Hence, \(x_\alpha'\) is
  in the same part of \(\RR\) as \(x_\alpha\)
  and \(\dist_{T_j}(x_\alpha', x_\alpha) \le 1\).
  In order to apply \cref{SetsMi}\ref{LocalGeodesicity} 
  to \(x_0'\) and \(x_q'\), observe that
  \begin{align*}
    \dist_{B_j - A_j}(x_0', x_q')&\le
    \dist_{T_j}(x_0', x_0) +
    \dist_Q(x_0, x_q) +
    \dist_{T_j}(x_q, x_q')\\
    &\le q + 2\\
    &\le \ell+3\\
    &< (8c+12)\Delta^{c+2} + 2c.
  \end{align*}
  Furthermore, since \(Q\) is \(\RR\)-clean,
  the part of \(\RR\) containing \(x_0\) and \(x_0'\)
  is distinct from the part containing
  \(x_q\) and \(x_q'\), so
  \(E(x_0'T_j^0 x_q') \cap M_j \neq \emptyset\).
  Therefore, by \cref{SetsMi}\ref{LocalGeodesicity},
  \[
  \dist_{T_j^0}(x_0', x_q') =
  \dist_{B_j - A_j}(x_0', x_q') \le
    q+2.
  \]
  Since \(M_j \subseteq E(T_j^0)\), we have
  \(E(x_0 T_j x_q) \cap M_j = E(x_0' T_j^0 x_q') \cap M_j
  \neq \emptyset\).
  Let \(e \in E(x_0' T_j^0 x_q') \cap M_j\).
  The length of the path \(x_0' T_j^0 x_q'\) is at most
  \(q+2\), so for each \(\alpha \in \{0, q\}\) we have
  \(\dist_{T_j^0}(e, x_\alpha') \le q + 1\),
  and therefore
  \[
    \dist_{T_j}(e, x_\alpha) =
    \dist_{T_j}(e, x_\alpha') +
    \dist_{T_j}(x_\alpha', x_\alpha)
    \le (q+1) + 1
    = q + 2.\qedhere
  \]
\end{proof}
\begin{claim}\label{IntersectionTi}
    Let \(j \in \{1, \ldots, m\}\), and let \(Q = (x_0, \ldots, x_q)\) be an
    \(\RR\)-clean path in \(B_j\)
    with \(q \in \{0, \ldots, \ell+1\}\)
    that is internally disjoint from \(A_j\).
    Then \(|V(Q) \cap V(T_j)| \le 2\).
\end{claim}
\begin{proof} 
  Suppose to the contrary that there exist distinct vertices
  \(y_1, y_2, y_3 \in V(Q) \cap V(T_j)\).
  Since \(T_j \subseteq B_j - A_j\) and \(Q\) is internally
  disjoint from \(A_j\), each subpath of \(Q\)
  between two of the vertices \(y_1,y_2,y_3\)
  is contained in \(B_j - A_j\).
  By \cref{Distance} applied to \(y_1 Q y_2\),
  there exists an edge \(e \in E(y_1 T_j y_2) \cap M_j\)
  with \(\dist_{B_j - A_j}(e, y_1) \le q+2\) and
  \(\dist_{B_j - A_j}(e, y_2) \le q+2\). 
  Without loss of generality, \(y_3\) belongs to the same
  component of \(T_j - e\) as \(y_1\), and therefore
  \(e \not \in E(y_1 T_j y_3)\).
  By \cref{Distance} applied to \(y_1 Q y_3\),
  there exists an edge \(e' \in E(y_1 T_j y_3) \cap M_j\)
  with \(\dist_{B_j - A_j}(e', y_1) \le q + 2\).
  Therefore, \(e \neq e'\), and
  \[
    \dist_{B_j - A_j}(e, e') \le
    \dist_{B_j - A_j}(e, y_1) +
    \dist_{B_j - A_j}(y_1, e') \le
    2(q+2) \le 2\ell+6=c,
  \]
  which contradicts \cref{SetsMi}\ref{DistanceDIndependent}.
\end{proof}

\begin{claim}\label{IntersectionNear}
    Let \(j \in \{1, \ldots, m\}\), and let
    \(Q = (x_0, \ldots, x_q)\) be an \(\RR\)-clean path
    in \(B_j\) with \(q \in \{0, \ldots, \ell+1\}\)
    that is internally disjoint from \(A_j\),
    such that \(x_0\) is an attachment-vertex of \(B_j\)
    on a tree \(T_i\) with \(i \in X_j\) and
    \(\dist_{B_i - A_i}(x_0, M_i) \le \ell+3\).
    Then \(|V(Q) \cap V(T_j)| \le 1\).
\end{claim}
\begin{proof}
    Suppose to the contrary that
    \(|V(Q) \cap V(T_j)| > 1\).
    Hence, by \cref{IntersectionTi}, we have
    \(|V(Q) \cap V(T_j)| = 2\), say
    \(V(Q) \cap V(T_j) = \{x_\alpha, x_\beta\}\)
    for some \(\alpha,\beta\in\{0, \ldots, q\}\)
    with \(\alpha < \beta\).
    Since \(x_0 \in V(T_i)\), we have
    \(\alpha \ge 1\).
    Since \(\{x_\alpha, x_\beta\} \subseteq V(T_j)
    \subseteq B_j - A_j\)
    and \(Q\) is internally disjoint from \(A_j\),
    we have \(x_1 Q x_\beta \subseteq B_j - A_j\).
    By \cref{Distance} applied to \(x_\alpha Q x_\beta\), we have
    \[
      \dist_{B_j - A_j}(x_\alpha, M_j) \le
      \beta - \alpha + 2
      \le q - \alpha + 2 \le \ell - \alpha+3.
    \]

    The vertex \(x_0\) is an attachment-vertex of \(B_j\) in \(V(T_i)\), \(i \in X_j\),
    and by \cref{FjBridgeInBj}, we have
    \(B_j \subseteq B_i\), and
    by \cref{NoAttachmentOnTj0},
    \(B_j\) has no attachment-vertices
    in \(V(T_i^0)\), so
    \(B_j - A_j \subseteq
    B_j \subseteq B_i - V(T_i^0)\).
    Therefore, 
    \begin{align*}       
    \mdist_i(M_j) &\le \dist_{B_i - A_i}(M_i, x_0) +
    \dist_{B_i - V(T_i^0)}(x_0, M_j)\\
    &\le (\ell+3) + (\dist_{B_j}(x_0, x_\alpha) + \dist_{B_j - A_j}(x_\alpha, M_j))\\
    &\le (\ell+3)+\alpha+(\ell-\alpha+3)
    = 2\ell+6=c.
    \end{align*}
    This contradicts \cref{SetsMi}\ref{DistanceMiMj}.
\end{proof}

Next, we bound the length of \(\RR\)-clean paths in some special cases. Recall that $F_j=\bigcup_{i<j} T_i$.

\begin{claim}\label{LengthTiTj}
  Let \(i, j \in \{1, \ldots, m\}\) with \(i < j\),
  and let \(Q = x_0 \cdots x_q\) be an \(\RR\)-clean path with
  \(q \in \{1, \ldots, \ell+1\}\),
  \(x_0 \in V(T_i)\), \(x_q \in V(T_j)\), and
  \(V(Q) \cap V(F_j) = \{x_0, x_q\}\).
  Then
  \begin{enumerate}[(\alph*)]
    \item\label{TiTj2} if \(\dist_{B_i - A_i}(x_0, M_i) \le \ell+3\) or \(\dist_{B_j - A_j}(x_q, M_j) \le
    \ell+3\), then \(q \le 2\);
    \item\label{TiTj4} otherwise, \(q \le 4\).
  \end{enumerate}
\end{claim}

\begin{proof}
Since \(V(Q) \cap V(F_j) = \{x_0, x_q\}\), the path
\(Q\)
is contained in some \(F_j\)-bridge.
If that \(F_j\)-bridge is trivial, then
\(q = 1\) and the claim follows.
Hence, \(Q\) is contained in a non-trivial
\(F_j\)-bridge, which equals \(B_k\) for some
\(k > j\) by \cref{FiBridgeIsBj}.
By \cref{VTj0Separates},
the set \(V(T_k^0)\) separates the vertices
\(x_0\) and \(x_q\) in \(B_k\), 
so some inner vertex of \(Q\) must lie on \(T_k^0\).

Let \(x_{\alpha}\) be an inner vertex of
\(Q\) in \(V(T_k^0)\).
Since the
vertices \(x_{\alpha-1}\) and \(x_{\alpha+1}\) are adjacent to \(V(T_k^0)\) in \(B_k\), they
 belong to \(A_k \cup V(T_k)\).
Since the only attachment-vertices of \(B_k\) on
\(Q\) are \(x_0 \in V(T_i)\) and \(x_q \in V(T_j\)),
we conclude that \(x_{\alpha-1} \in \{x_0\} \cup V(T_k)\)
and \(x_{\alpha+1} \in \{x_q\} \cup V(T_k)\).
By \cref{IntersectionTi},
we have \(|V(Q) \cap V(T_k)| \le 2\).
Since \(x_\alpha \in V(T_k^0) \subseteq V(T_k)\),
at most one of the vertices \(x_{\alpha-1}\)
and \(x_{\alpha+1}\) lies on \(T_k\).
In particular, \(x_{\alpha-1} = x_0\) or
\(x_{\alpha+1} = x_q\), so \(\alpha \in \{1, q-1\}\).
If \(x_{\alpha-1}=x_0\) and \(x_{\alpha+1} = x_q\),
then \(\alpha=1\) and \(q=2\), and the claim holds.
Hence we may assume that one of \(x_{\alpha-1}\) and \(x_{\alpha+1}\) lies on \(T_k\), and
therefore
\(V(Q) \cap V(T_k) = \{x_1, x_2\}\) or
\(V(Q) \cap V(T_k) = \{x_{q-2}, x_{q-1}\}\).
Thus, there exists \(\beta \in \{1, q-2\}\) such that
\(V(Q) \cap V(T_k) = \{x_\beta, x_{\beta+1}\}\).
By \cref{Distance} applied to the path
\(x_\beta Q x_{\beta+1}\), we have
\(\dist_{B_k - A_k}(x_\beta, M_k) \le 3\) and
\(\dist_{B_k - A_k}(x_{\beta+1}, M_k) \le 3\).

For the proof of \cref{TiTj2}, suppose that
\(\dist_{B_i - A_i}(x_0, M_i) \le \ell+3\) or
\(\dist_{B_j - A_j}(x_q, M_j) \le \ell+3\).
Hence, by \cref{IntersectionNear},
we have \(|V(Q) \cap V(T_k)| = 1\), so
\(x_{\alpha-1} = x_0\) and \(x_{\alpha+1} = x_q\), and therefore
\(q = 2\).
This proves \cref{TiTj2}.

Next, we show \cref{TiTj4}.
Suppose that \(\beta = 1\).
Then
\(x_q \in V(T_j)\), \(x_2 \in V(T_k)\),
\(V(x_2 Q x_q) \cap V(F_k) = \{x_q, x_2\}\),
and \(\dist_{B_k - A_k}(x_2, M_k) \le 3 < \ell+3\).
Hence, by \ref{TiTj2}, the length of \(x_2 Q x_q\)
is at most \(2\), and therefore \(q \le 4\).
The case when \(\beta = q-2\) is similar:
We have \(x_0 \in V(T_i)\), \(x_{q-2} \in V(T_k)\),
\(V(x_0 Q x_{q-2}) \cap V(F_k) = \{x_0, x_{q-2}\}\),
and \(\dist_{B_k - A_k}(x_{q-2}, (M_k)) \le 3 < \ell+3\).
Hence, by \ref{TiTj2}, the length of \(x_0 Q x_{q-2}\)
is at most \(2\), so \(q \le 4\).
This completes the proof of \ref{TiTj4}.
\end{proof}

\begin{claim}\label{LengthTiTi}
    Let \(i \in \{1, \ldots, m\}\),
    and let \(Q = x_0 \cdots x_q\) be an \(\RR\)-clean path
    with \(q \in \{0, \ldots, \ell+1\}\) and
    \(\{x_0, x_q\} \subseteq 
    V(Q) \cap V(F_i) \subseteq V(T_i)\).
    Then \(q \le 4\).
\end{claim}
\begin{proof}
  If \(q = 0\), then the claim holds trivially, so we
  assume that \(x_0 \neq x_q\).
  By \cref{IntersectionTi}, we have
  \(|V(Q) \cap V(F_i)| \le 2\), so
  \(V(Q) \cap V(F_i) = \{x_0, x_q\}\).
  Since $(V(T_j)\colon j\in \{1,\dots,m\})$ is a partition of $V(G)$, each inner vertex of \(Q\) belongs to some tree \(T_j\) with \(j > i\). 
  Let \(T_j\) be the tree
  containing an inner vertex
  of \(Q\) with the smallest
  \(j\).  
  Thus, \(Q\) intersects \(F_{j-1}\) only in its ends, and \(B_j\) is the \(F_{j-1}\)-bridge
  containing \(Q\). 
  By \cref{FjBridgeInBj},
  \(B_j \subseteq B_i\).
  Since \(B_j\) has attachment-vertices in \(V(T_i)\), we have \(i \in X_j\).
  By \cref{Distance}, we have
  \(\dist_{T_i}(x_0, M_i) \le q + 2 \le \ell+3\) and
  \(\dist_{T_i}(x_q, M_i) \le q + 2 \le \ell+3\).
  Hence, by \cref{IntersectionNear}, we have
  \(|V(Q) \cap V(T_j)| = 1\), say
  \(V(Q) \cap V(T_j) = \{x_\alpha\}\). 
  By \cref{LengthTiTj}\cref{TiTj2}, each of the paths \(x_0 Q x_\alpha\) and
  \(x_\alpha Q x_q\) has length at most \(2\), so
  \(q\le 4\).
\end{proof}

For each \(j \in \{1, \ldots, m\}\),
the graph \(B_j\) intersects at most
three components of \(F_j\), namely,
\(T_j\) and at most two components of
\(F_{j-1}\) on which \(B_j\) has attachment-vertices.
We aim to show that every \(\RR\)-clean
path in \(B_j\) with both ends on
\(F_j\) has length at most \(36\)
(the value \(\tau = 37\) 
was chosen to be greater than this bound).
We first prove the following helper
claim.

\begin{claim}\label{Length8a}
    Let \(i, j \in \{1, \ldots, m\}\)
    with \(i < j\),
    let \(Q = (x_0, \ldots, x_q)\) be
    an \(\RR\)-clean path in \(B_j\)
    with \(q \in \{0, \ldots, 
    \ell+1\}\),
    \(V(Q) \cap V(T_i) = \{x_0\}\)
    and \(x_q \in V(F_j)\).
    Then \(q \le 8(a-1)\),
    where 
    \(a \in \{1, 2, 3\}\)
    is the number of components of
    \(F_j\) that intersect \(Q\).
\end{claim}
\begin{proof}
    If \(a = 1\), then since
    \(V(Q) \cap V(T_i) = \{x_0\}\), the
    only component of \(F_j\) intersecting
    \(Q\) is \(T_i\), and thus \(x_q = x_0\),
    so \(q=0=8(a-1)\). 
    Hence, we assume that
    \(a \ge 2\).
    Let \(T_{i'}\) be the component of
    \(F_j\) that contains a vertex
    of \(Q\), is distinct from \(T_i\)
    and has \(i'\) as small as possible.
    Let \(x_\alpha\) and \(x_\beta\)
    denote respectively the first and
    the last vertex of \(Q\) in \(V(T_{i'})\).
    By \cref{LengthTiTj}, the length of
    \(x_0 Q x_\alpha\) is at most \(4\),
    and by \cref{LengthTiTi}, the length
    of \(x_\alpha Q x_{\beta}\) is also
    at most \(4\). Hence, \(\beta \le 8\).
    If \(i' = j\), then by our choice of \(i'\), we have \(V(Q) \cap V(F_j) \subseteq \{x_0\} \cup V(T_j)\),
    so \(x_q \in V(T_j)\), and therefore \(q = \beta \le 8 \le 8(a-1)\).
    Hence, assume that \(i' \neq j\) which means that \(i' \in X_j\).
    The path \(Q\) intersects \(T_i\) and \(T_{i'}\),
    and by \cref{VTj0Separates} it intersects also
    \(T_j\), so \(a = 3\).
    We have \(V(x_\beta Q x_q) \cap V(T_{i'}) = \{x_\beta\}\),
    and the path \(x_\beta Q x_q\) intersects at most
    two components of
    \(F_j\), so we already know that 
    its length is at most \(8(2-1) = 8\),
    so \(q \le \beta + 8 \le 16 = 8(a-1)\).
\end{proof}
\begin{claim}\label{Length36}
    Let \(j \in \{1, \ldots, m\}\), and let
    \(Q = (x_0, \ldots, x_q)\) be an
    \(\RR\)-clean path in \(B_j\) with
    \(p \in \{0, \ldots, \ell+1\}\) 
    such that
    \(\{x_0, x_p\} \subseteq V(F_j)\).
    Then \(q \le 36\).
\end{claim}
\begin{proof}
  Let \(T_i\) be the tree 
  intersecting \(Q\) with the smallest \(i\). 
  So $i\in X_j\cup \{j\}$.
  Let \(x_\alpha\) and \(x_\beta\)
  denote the first and the last vertex of
  \(Q\) in \(V(T_i)\).
  By \cref{LengthTiTi}, the length of
  \(x_\alpha Q x_\beta\) is at most \(4\).
  If \(i = j\), then \(\alpha = 0\) and
  \(\beta = q\), so \(q \le 4\).
  Therefore, we assume that \(i < j\).
  We have
  \(V(x_0 Q x_{\alpha}) \cap V(T_i)=
  \{x_\alpha\}\) and
  \(V(x_\beta Q x_q) \cap V(T_i) = \{x_\beta\}\).
  Hence, by \cref{Length8a}, each of the paths \(x_0 Q x_\alpha\) and
  \(x_\beta Q x_q\) has length at most \(16\),
  which implies that \(q \le 16 + 4 + 16 = 36\).
\end{proof}

We proceed to the main part of the proof of \cref{BlockingPlanar}.
Towards a contradiction, assume that
\(\RR\) is not \(\ell\)-blocking. Hence, there
exists an \(\RR\)-clean path \(P = (x_0, \ldots, x_p)\)
with \(p > \ell\).
Every subpath of an \(\RR\)-clean path is \(\RR\)-clean,
so we may assume without loss of generality that the length of \(P\)
is exactly \(\ell+1\).
Let
\(T_i\) be the tree
intersecting \(Q\)
that has the smallest \(i\).
Let \(x_\alpha\) and
\(x_\beta\) denote respectively the first and the last vertex of \(Q\)
belonging to \(T_i\).
By \cref{LengthTiTi}, the length of
\(x_\alpha P x_\beta\) is at most \(4\).
Hence, there exists \(Q \in \{x_0 P x_\alpha, x_\beta P x_p\}\)
with length at least
\(\lceil((\ell+1)-4)/2\rceil = 110\).
The path \(Q\) intersect \(V(F_i)\) only in one of its ends,
and that end lies on \(T_i\).
Therefore, to reach a contradiction and complete the proof
it suffices to show the following claim.

\begin{claim}\label{LengthRooted}
    Let \(i \in \{1, \ldots, m\}\), and let \(Q = (x_0, \ldots, x_q)\) be an
    \(\RR\)-clean path with
    \(q \in \{0, \ldots, \ell+1\}\) and
    \(V(Q) \cap V(F_i) = \{x_0\} \subseteq V(T_i)\).
    Then \(q \le 109\).
\end{claim}

\cref{LengthRooted} is a consequence of the following technical claim.
\begin{claim}\label{Reduction}
  Let \(i \in \{1, \ldots, m\}\),
  let \(Q = (x_0, \ldots, x_q)\) be an \(\RR\)-clean path contained in an \(F_i\)-bridge such that \(q \in \{74, \ldots, \ell+1\}\) 
  and \(V(Q) \cap V(T_i) = \{x_0\}\).
  Then
  \begin{enumerate}[(\alph*)]
      \item\label{ReductionA}
      there exist \(i' \in \{1, \ldots, m\}\)
      and an \(\RR\)-clean path \(Q' = (x_0', \ldots, x_{q'}')\)
      contained in an \(F_{i'}\)-bridge
      such that \(q' \in \{q-36, \ldots, \ell+1\}\), 
      \(V(Q') \cap V(T_{i'}) = \{x_0'\}\) and
      \(\mdist_{i'}(x_0') \le 39\), and
      \item\label{ReductionB}
      there exist \(j \in \{i+1, \ldots, m\}\) such that
      \(i \in X_j\) and
      \(\dist_{B_i - V(T_{i}^0)}(x_0, M_j) \le 42\).
  \end{enumerate}
\end{claim}

Before proving \cref{Reduction}, we show how it implies \cref{LengthRooted}.

\begin{proof}[Proof of \cref{LengthRooted} assuming \cref{Reduction}]    
  Towards a contradiction, suppose that \(q \ge 110\).
  We will apply \cref{Reduction}\cref{ReductionA} to \((i, Q)\) to obtain
  a pair \((i', Q')\), and then we will
  apply \cref{Reduction}\cref{ReductionB}
  to \((i', Q')\) to obtain an index \(j'\)
  with contradictory properties.

  We have \(q \ge 110\), so
  in particular, \(q \in \{74, \ldots, \ell+1\}\).
  Since \(V(Q) \cap V(F_i) = \{x_0\} \subseteq V(T_i)\),
  we have \(V(Q) \cap V(T_i) = \{x_0\}\) and the path \(Q\) is contained
  in an \(F_i\)-bridge, so \(i\) and \(Q\) satisfy the preconditions
  of \cref{Reduction}.
  By \cref{Reduction}\ref{ReductionA}, there exist
  \(i' \in \{1, \ldots, m\}\) and an \(\RR\)-clean path
  \(Q' = (x_0', \ldots, x_{q'}')\)
  contained in an \(F_{i'}\)-bridge such that
  \(q' \in \{q-36, \ldots, q\} \subseteq \{74, \ldots, \ell+1\}\),
  \(V(Q) \cap V(T_{i'}) = \{x_0'\}\) and \(\mdist_{i'}(x_0') \le 39\).
  Hence, \(i'\) and \(Q'\) satisfy the preconditions
  of \cref{Reduction}.
  By \cref{Reduction}\ref{ReductionB} applied to \(i'\) and \(Q'\), there exist
  \(j' \in \{i'+1, \ldots, m\}\) such that
  \(i' \in X_{j'}\) and
  \(\dist_{B_{i'} - V(T_{i'}^0)}(x_0', M_j) \le 42\).
  Therefore,
  \begin{align*}
    \mdist_{i'}(M_{j'}) \le \mdist_{i'}(x_0')
    + \dist_{B_{i'} - V(T_{i'}^0)}(x_0', M_{j'}) \le
    39 + 42 = 81,
  \end{align*}
  which contradicts \cref{SetsMi}\ref{DistanceMiMj} since $c=450>81$.
\end{proof}

The proof of \cref{Reduction} makes use of the following claim:
\begin{claim}\label{VerySpecialCase}
  Let \(j \in \{1, \ldots, m\}\),
  let \(Q = (x_0, \ldots, x_q)\) be an 
  \(\RR\)-clean path in \(B_j - V(T_j^0)\) with
  \(q \in \{0, \ldots, \ell+1\}\) such that
  \(x_0 \in A_j\),
  and there exists an \(F_j\)-bridge \(B\) contained
  in \(B_j\) with \(x_0 \in V(B)\) and \(x_q \not \in V(B)\).
  Then \(q \le 37\).
\end{claim}

\begin{proof}[Proof]
  Let \(x_\alpha\) be the last vertex
  on \(Q\) that belongs to \(V(F_j)\).
  By \cref{Length36}, we have
  \(\alpha \le 36\).
  Unless \(q = \alpha \le 36\), the path
  \(x_\alpha Q x_q\) is contained in a
  non-trivial \(F_k\)-bridge which equals
  \(B_k\) for some \(k \in \{j+1, \ldots, m\}\) by \cref{FiBridgeIsBj}.
  Let \(x_{\alpha'}\) be the first vertex
  of \(Q\) that belongs to \(V(B_k)\).
  By \cref{BBkPathAkout}, we have
  \(x_{\alpha'} \in A_k^{\out}\).
  Towards a contradiction, suppose that
  \(q \ge 38\), and thus \(q \ge \alpha+2\).
  Since
  \(\dist_G(x_{\alpha'}, x_\alpha) \le 36\), we have 
  \(x_{\alpha+1} \in D_k \subseteq V(T_k^0)\),
  and therefore \(x_{\alpha+2} \in V(T_k)\).
  By definition of \(T_k\), there exists a
  vertex \(x_{\alpha+2}' \in V(T_k^0)\) 
  that belongs
  to the same component of \(T_k - M_k\) as \(x_{\alpha+2}\) and satisfies
  \(\dist_{T_k}(x_{\alpha+2}, x_{\alpha+2}') \le 1\). In particular,
  \(\dist_{B_k - A_k}(x_{\alpha+1}, x_{\alpha+2}') \le 2\).
  By \cref{ApproxGeodesicTj0}, the length
  of the path \(x_{\alpha+1} T_k^0 x_{\alpha+2}'\) is less than
  \(2\Delta^{40}\).
  Since \(Q\) is \(\RR\)-clean,
  we have \(E(x_{\alpha+1}T_k^0x_{\alpha+2}') \cap M_k \neq \emptyset\), and therefore
  \(\dist_{T_j^0}(D_k, M_k) < 2\Delta^{40}\),
  which contradicts
  \cref{SetsMi}\ref{DistanceDiMi}.
\end{proof}

It remains to prove \cref{Reduction}.

\begin{proof}[Proof of \cref{Reduction}]
  Since \(V(Q) \cap V(F_i) = \{x_0\}\) and \(q \ge 74 > 1\),
  \(Q\) is contained in a non-trivial \(F_i\)-bridge,
  so by \cref{NoAttachmentOnTj0}, we have
  \(x_0 \not \in V(T_i^0)\), and by
  \cref{FiBridgeIsBj}, there exists \(j \in \{i+1, \ldots, m\}\)
  with \(Q \subseteq B_j\).
  Fix the largest \(j \in \{i+1, \ldots, m\}\) with \(Q \subseteq B_j\).
  We split the argument into two cases based on
  whether \(Q\) intersects \(T_j^0\) or not.

  \textit{Case 1.} \(V(Q) \cap V(T_j^0) = \emptyset\).
  Let \(B\) be the \(F_j\)-bridge containing the edge \(x_0 x_1\). Hence,
  \(B \subseteq B_j\),
  and \(B\) has attachment-vertices in \(V(T_i)\) and
  \(V(T_j)\).
  We have \(x_q \in V(B)\) because otherwise
  \cref{VerySpecialCase} would imply
  \(q \le 37\) contrary to our assumption that \(q \ge 74\).
  Therefore \(\{x_0, x_1, x_q\} \subseteq V(B)\),
  and in particular, \(B\) is non-trivial.
  By \cref{FiBridgeIsBj}, we have \(B = B_k\) for some \(k \in \{j+1, \ldots, m\}\).
  By our choice of \(j\), we have
  \(Q \not \subseteq B_k\).
  Hence there exist
  \(\alpha, \beta \in \{0, \ldots, q\}\)
  with \(\alpha < \beta\) 
  such that
  \(\{x_\alpha, x_\beta\} \subseteq V(B_k)\), \(x_\alpha Q x_\beta\) is
  edge-disjoint from \(B_k\) and
  \(x_\beta Q x_q \subseteq B_k\).
  Since \(x_0 x_1 \in E(B_k)\), we have
  \(\alpha > 0\).
  The vertices \(x_\alpha\) and \(x_\beta\)
  are attachment-vertices of \(B_k\).
  Since \(x_0\) is the only vertex of \(Q\)
 in \(V(T_i)\), the vertices \(x_\alpha\) and
  \(x_\beta\) lie on \(T_j\).
  By \cref{LengthTiTi}, the length of
  \(x_\alpha Q x_\beta\) is at most \(4\),
  and we have
  \(\dist_{B_j - A_j}(x_\beta, M_j) \le 6\)
  by \cref{Distance}.
  In particular, \(\mdist_{j}(x_\beta) \le 6\).
  By \cref{Length36}, we have \(\beta \le 36\),
  so \ref{ReductionA} is satisfied by
  \(i' = j\) and \(Q' = x_\beta Q x_q\).
  Furthermore, 
  \(\dist_{B_i - V(T_i^0)}(x_0, M_j) \le \beta + 6 \le 42\), so \(j\) satisfies \ref{ReductionB}.

  \textit{Case 2.} \(V(Q) \cap V(T_j^0) \neq \emptyset\).
  Let \(x_\alpha\) be the last vertex of \(Q\)
  in \(V(T_j^0)\).
  By \cref{Length36}, we have \(\alpha \le 36\).
  Since \(x_{\alpha} \in V(T_j^0)\), we 
  have
  \(x_{\alpha+1} \in A_j \cup V(T_j)\).
  Suppose towards a contradiction,
  that \(x_{\alpha+1} \in A_j\). 
  By \cref{Length36}, we have \(\alpha + 1 \le 36\).
  Then \(x_\alpha x_{\alpha+1}\) is a trivial
  \(F_j\)-bridge contained in \(B_j\)
  that contains \(x_{\alpha+1}\) and does not
  contain \(x_q\). By \cref{VerySpecialCase} applied
  to \(x_{\alpha+1} Q x_q\), the length of
  \(x_{\alpha+1} Q x_q\) is at most \(37\), so
  \(q \le (\alpha + 1) + 37 \le 36 + 37 < 74\),
  a contradiction. Therefore,
  \(x_{\alpha+1} \not \in A_j\), so
  \(x_{\alpha+1} \in V(T_j)\).
  By \cref{Distance} applied to \(x_\alpha Q x_{\alpha+1}\), we have
  \(\dist_{B_j - A_j}(x_{\alpha + 1}, M_j) \le 3\),
  and thus \(\dist_{B_i - V(T_i^0)}(x_0, M_j) \le
  (\alpha+1) + 3 \le 36 + 3 = 39\).
  This proves \ref{ReductionB}.
  
  For the proof of \ref{ReductionA},
  let \(x_\beta\) and \(x_\gamma\) denote the
  last two vertices of \(Q\) in \(V(F_j)\) where
  \(\beta < \gamma\).
  Since \(\{x_\alpha, x_{\alpha+1}\} \subseteq V(T_j)\), we have \(\beta \ge \alpha\),
  and by \cref{Length36}, we have \(\gamma \le 36\).
  We have \(\{x_\beta, x_\gamma\} \subseteq V(B_j) \cap V(F_j) = A_j \cup V(T_j)\).
  We consider three subcases.
  
  \textit{Subcase 2.1.} \(x_\gamma \in V(T_j)\).
  By definition of \(x_\alpha\), the path
  \(x_{\alpha+1} Q x_\gamma\) is disjoint from
  \(T_j^0\), so \(x_{\alpha+1} \in V(T_j) \setminus V(T_j^0)\) and \cref{ReductionA} is satisfied by
  \(i' = j\) and \(Q' = x_\gamma Q x_q\) since
  \[
  \mdist_j(x_\gamma) \le
  \dist_{B_j - A_j}(M_j, x_{\alpha+1}) +
  \dist_{B_j - V(T_j^0)}(x_{\alpha+1}, x_\gamma)
  \le \alpha + 1 + 3 \le 36 + 3 = 39.
  \]

  \textit{Subcase 2.2.} \(x_\gamma \in A_j\) and
  \(x_\beta \in V(T_j)\).
  The path \(x_\beta Q x_\gamma\) is internally
  disjoint from \(F_j\), so it is contained
  in an \(F_j\)-bridge \(B\) such that
  \(B \subseteq B_j\).
  We have \(x_q \in V(B)\), since otherwise by
  \cref{VerySpecialCase} applied to
  \(x_\gamma Q x_q\) the length of \(x_\gamma Q x_q\)
  is at most \(37\), so \(q \le \gamma + 37 \le 36 + 37 < 74\), which is a contradiction.
  Hence, \(x_q \in V(B)\).
  Therefore, \(B\) is non-trivial, and by
  \cref{NoAttachmentOnTj0}, we have
  \(B \subseteq B_j - V(T_j^0)\).
  Since \(\gamma < q\), \(x_q\) is not an
  attachment-vertex of \(B\), and we have
  \(x_\gamma Q x_q \subseteq B\), so
  \(x_\beta Q x_q \subseteq B \subseteq B_j - V(T_j^0)\).
  Thus, \cref{ReductionA} is satisfied by
  \(i' = j\) and \(Q' = x_\beta Q x_q\)
  since
  \[
  \mdist_j(x_\beta) \le 
  \dist_{B_j - A_j}(M_j, x_{\alpha+1}) +
  \dist_{B_j - V(T_j^0)}(x_{\alpha+1}, x_\beta)
  \le \alpha + 1 + 3 \le 36 + 3 = 39.
  \]

  \textit{Subcase 2.3} \(x_\gamma \in A_j\) and
  \(x_\beta \in A_j\).

  By the invariant from the
  construction of the forests \(F_j\), 
  the \(F_{j-1}\)-bridge \(B_j\) has attachments
  on exactly one tree \(T_{i'}\) with \(i' \neq i\),
  and the vertices \(x_\gamma\) and \(x_\beta\)
  lie on that tree \(T_{i'}\)
  because the only vertex of \(Q\) on \(T_i\) is \(x_0\).
  By \cref{LengthTiTi}, the length of
  \(x_\beta Q x_\gamma\) is at most \(4\), and
  by \cref{Distance}, we have
  \(\dist_{B_{i'} - A_{i'}}(M_{i'}, x_\gamma) \le 6\).
  In particular, \(\mdist_{i'}(x_\gamma) \le 6\).
  Hence, \ref{ReductionA} is satisfied by \(i'\)
  and \(Q' = x_\gamma Q x_q\).
  This completes the proof.
\end{proof}

\section{Graphs on Surfaces}
\label{Surfaces}

This section proves \cref{BlockingGenus} which lifts our result for blocking partitions of planar graphs (\cref{BlockingPlanar}) to graphs on surfaces. We need the following folklore lemma (implicit in \cite{DMW17,DHHW22} for example).

\begin{lem}\label{TreeGenus}
    For every connected graph $G$ with Euler genus $g$ and for every
    \textsc{bfs}-layering $(V_0,V_1,\dots)$ of $G$, $G$ contains a tree $T$ that is the union of at most $2g$ vertical paths with respect to $(V_0,V_1,\dots)$ such that $G-V(T)$ is planar.
\end{lem}

The next lemma is stated in terms of the following subgraph variant of clean paths: Let $G$ be a graph and $\mathcal{Z}$ be a connected partition of a subgraph $Z$ of $G$. A path $P$ in $G$ is \defn{$\mathcal{Z}$-clean} if $|V(P)\cap V|\leq 1$ for each $V\in \mathcal{Z}$. 

\begin{lem}
\label{GenusZ}
For any integers $g\geq 0$ and $\ell\geq 1$, every connected graph $G$ with Euler genus $g$ has a connected subgraph $Z$ such that $G-V(Z)$ is planar and $Z$ has a connected partition $\mathcal{Z}$ with width at most $2g((5g+1)\ell+3)$ such that every  $\mathcal{Z}$-clean path of length at most $\ell$ in $G$ intersects at most three parts in $\mathcal{Z}$. 
\end{lem}

\begin{proof}
The $g=0$ case holds trivially with $Z$ the empty graph and $\mathcal{Z}$ the empty set. Now assume that $g\geq 1$. Let $(V_0,V_1,\dots)$ be a
\textsc{bfs}-layering of $G$ where $V_0=\{r\}$ for some $r\in V(G)$. By \cref{TreeGenus}, $G$ contains a tree $T$ that is the union of at most $2g$ vertical paths such that $G':=G-V(T)$ is planar. For $a,b\in \NN_0$ where $a \leq b$, let $V_{[a,b]}:=\bigcup(V_j \colon j\in\{a,\dots,b\})$ and $T_{[a,b]}:=T[V(T)\cap V_{[a,b]}]$. For $i\in\NN_0$, we inductively construct a sequence of tuples $(x_i,X_i,Z_i, \mathcal{X}_i, \mathcal{Z}_i)$ with the following properties:
      \begin{enumerate}[(1)]
          \item\label{Tuple1} $x_0=0$ and $x_i\in \{x_{i-1}+3g\ell+1,\dots, x_{i-1}+ 5g\ell+1\}$ for all $i\geq 1$;    
          \item\label{Tuple2}  $X_i$ is an induced subgraph of $G$ with $V(T_{[x_{i-1}+1,x_i]}) \subseteq V(X_i)\subseteq V_{[x_{i-2}+\ell+1,x_i]}$;
          \item\label{Tuple3}  $Z_i$ is an induced subgraph of $G$ with $V(T_{[0,x_i]}) \subseteq V(Z_i)\subseteq V_{[0,x_i]}$; 
          \item\label{Tuple4}  $X_i=Z_i-V(Z_{i-1})$;
          \item\label{Tuple5}  $\mathcal{X}_i$ is a connected partition of $X_i$ of width at most $2g((5g+1)\ell+3)$; 
          \item\label{Tuple6}  $\mathcal{Z}_i$ is a connected partition of $Z_i$ of width at most $2g((5g+1)\ell+3)$ where $\mathcal{Z}_{i}=\mathcal{Z}_{i-1}\cup \mathcal{X}_i$; 
          \item\label{Tuple7}  Every path in $G-V(Z_{i-1})$ of length at most $\ell$ intersects at most one part in $\mathcal{X}_i$; and
          \item\label{Tuple8}  Every $\mathcal{Z}_i$-clean path in $G$ of length at most $\ell$ intersects at most three parts in $\mathcal{Z}_i$.
      \end{enumerate}
      Note that when $i:=|V(G)|$, we have $x_i\geq |V(G)|$ and $T\subseteq Z_i$, which implies that $(Z_i,\mathcal{Z}_i)$ satisfies the lemma statement. 
      
     For $i=0$, such a tuple exists with $X_i=G[\{r\}]$, $Z_i=X_i$, $\mathcal{X}_i=(\{r\})$, and $\mathcal{Z}_i=\mathcal{X}_i$. Now assume that $i\geq 1$ and such a tuple exists for $i-1$.

      Let $x_{i,1}:=x_{i-1}+3g\ell+1$ and $X_{i,1}:= T_{[x_{i-1}+1,x_{i,1}]}$. Then $X_{i,1}$ is the union of at most $2g$ vertical paths, and thus 
      has at most $2g$ components. Suppose $G-V(Z_{i-1})$ contains a path $P$ of length at most $\ell$ that intersects at least two components of $X_{i,1}$. Let $x_{i,2}:=\max\{\{j\colon V(P)\cap V_j\neq \emptyset\}\cup x_{i,1}\}$ and  $X_{i,2}:= G[V(T_{[x_{i-1}+1,x_{i,2}]})\cup V(P)]$. Then $X_{i,2}$ has at most $2g-1$ components. Moreover, since $P$ has length at most $\ell$, it follows that  $x_{i,2}\in \{x_{i,1},\dots, x_{i,1}+\ell \}$ and  $V(P)\subseteq V_{[x_{i-1}-\ell+1,x_{i,1}]}$, so $V(X_{i,2})\subseteq V_{[x_{i-1}-\ell+1,x_{i,1}]}$. Iterate the above procedure until there is no path of length at most $\ell$ that intersects two components of $X_{i,j}$. Such a process must terminate within at most $2g$ iterations, since no path can exist if $X_{i,j}$ has only one component. As such, there exists $j\in \{1,\dots,2g\}$ such that $x_{i,j}\in \{x_{i-1}+3g\ell+1,\dots, x_{i-1}+5g\ell+1\}$, $V(X_{i,j})\subseteq V_{[x_{i-1}-2g\ell+1,x_{i,1}]}\subseteq V_{[x_{i-2}+\ell+1,x_{i,1}]}$ and every path in $G-V(Z_{i-1})$ of length at most $\ell$ intersects at most one component of $X_{i,j}$. Set $x_i:=x_{i,j}$, $X_i:=G[V(X_{i,j})]$, and $Z_i:=G[Z_{i-1}\cup X_i]$.  Let $\mathcal{X}_i$ be the connected partition of $X_i$ where each part induces a component of $X_i$ and $\mathcal{Z}_i:=\mathcal{Z}_{i-1}\cup X_{i,j}$. We now show that the construction satisfies the desired properties.
        
      By construction, \ref{Tuple1}, \ref{Tuple2}, \ref{Tuple3}, \ref{Tuple4} and \ref{Tuple7} hold clearly. For \ref{Tuple5}, since $X_{i}$ is the union of at most $2g$ vertical paths of length at most $5g\ell+1$ together with the union of at most $2g$ paths of length at most $\ell$, it follows that $|V(X_{i})|\leq  2g((5g+1)\ell+3)$. Thus \ref{Tuple5} holds and so, by induction, \ref{Tuple6} holds. It remains to show \ref{Tuple8}. Let $P$ be a $\mathcal{Z}_i$-clean path in $G$ of length at most $\ell$. If $V(P)\subseteq V_{[0,x_{i-2}+\ell]}$, then the claim follows by induction. So assume that $V(P)\cap V_{[x_{i-2}+\ell+1,x_i]}\neq \emptyset$. Since $V(Z_{i-2}) \subseteq V_{[0, x_{i-2}]}$, this implies $V(P)\cap V(Z_{i-2})=\emptyset$. Thus $P$ only intersects parts in $\mathcal{X}_{i-1}\cup \mathcal{X}_i$. As $P$ has length at most $\ell$, it follows by \ref{Tuple7} that $P$ only intersects at most one part of $\mathcal{X}_{i-1}$. Let $W:=V(P)\cap V(Z_{i-1})$. Since $P$ is $\mathcal{Z}_i$-clean, it follows that $|W|\leq 1$. So $P-W$ consists of at most two components that are $\mathcal{Z}_i$-clean paths in $G-V(Z_{i-1})$. By \ref{Tuple7}, each component of $P-W$ intersects at most one part in $\mathcal{X}_i$. So $P$ intersects at most three parts in $\mathcal{Z}_i$, as required.
  \end{proof}  

\begin{proof}[Proof of \cref{BlockingGenus}]
    Without loss of generality, we may assume that $G$ is connected. By \cref{GenusZ} with $\ell=895$, $G$ contains a subgraph $Z$ such that $G':=G-V(Z)$ is planar and $Z$ has a connected partition $\mathcal{Z}$ with width at most $8950g^2 + 1796g$ such that every path of length at most $895$ in $G$ intersects at most three parts in $\mathcal{Z}$. By \cref{BlockingPlanar}, $G'$ has a $\BlockingNumber$-blocking partition $\mathcal{R}'$ with width at most $10\Delta^{80}(3612\,\Delta^{452}+900)$. Let $\mathcal{R}:=\mathcal{R}'\cup \mathcal{Z}$, which is a connected partition of $G$ with width at most $\max\{10\Delta^{80}(3612\,\Delta^{452}+900),8950g^2 + 1796g\}$. We claim that $\mathcal{R}$ is $894$-blocking. Consider an $\mathcal{R}$-clean path $P$ in $G$. Then $P$ intersects at most three parts in $\mathcal{Z}$. Let $W:=V(P)\cap V(Z)$. Since $P$ is $\mathcal{R}$-clean, $|W|\leq 3$. Therefore, $P-W$ has at most four components, each of which is an $\mathcal{R}'$-clean path in $G'$. Since each $\mathcal{R}'$-clean path in $G'$ has length at most $\BlockingNumber$, it follows that $P$ has length at most $4\cdot\BlockingNumber+6=894$. Hence  $\mathcal{R}$ is $894$-blocking.
\end{proof}

\section{Reflections on Blocking Partitions}
\label{Reflections}

This section considers which graph classes have $\ell$-blocking partitions of width at most $c$ for some constants $\ell,c$. Bounded maximum degree is necessary, even for trees.

\begin{prop}
\label{trees}
If every tree with maximum degree $\Delta$ has an $\ell$-blocking partition of width at most $c$, then $c\geq \Delta$.
\end{prop}

\begin{proof}
Let $T$ be the complete  $(\Delta-1)$-ary rooted tree of height $\ell+1$. So $T$ has maximum degree $\Delta$ and every root-to-leaf path has length $\ell+1$.  Let $\RR$ be an $\ell$-blocking partition of $T$ of width at most $c$. For the sake of contradiction, suppose $c<\Delta$. Then every non-leaf vertex of $T$ has a child that belongs to a different part in $\RR$. So $T$ contains a root-to-leaf path $P$ where every pair of consecutive vertices belong to different parts in $\RR$. Moreover, no two non-consecutive vertices in $P$ belong to the same part, since each part in $\RR$ is connected. Hence $P$ is an $\RR$-clean path of length $\ell+1$, which is a contradiction.
\end{proof}

On the other hand, bounded maximum degree is not enough. 

\begin{prop}
\label{4regular}
    There are no constants $c,\ell\in\NN$ such that every 4-regular graph has an $\ell$-blocking partition of width at most $c$.
\end{prop}

\begin{proof}
Suppose for the sake of contradiction that every 4-regular graph has an $\ell$-blocking partition of width at most $c$. \citet{ES63} showed that for any integers $\Delta,g\geq 3$ there is a $\Delta$-regular graph with girth at least $g$. Let $G$ be a 4-regular $n$-vertex graph with girth $g\geq c+\ell+2$. Consider an $\ell$-blocking partition $\RR$ of $G$ with width at most $c$. Say that an edge $uv\in E(G)$ is \defn{red} if $u,v\in V$ for some $V\in \RR$, otherwise it is \defn{blue}. Since $g>c$, each part $V\in \RR$ induces a tree, so the total number of red edges is less than $n$. Thus the number of blue edges is more than $|E(G)|-n=n$. Hence there is a cycle $C$ in $G$ that consists of blue edges, which has length at least $g\geq \ell+2$. Therefore $C$ contains a path $P$ of length $\ell+1$ that consists of blue edges. If distinct vertices $v,w$ in $P$ are in the same part in $\RR$, then $G$ contains a cycle of length at most $c+\ell$, which contradicts the choice of $g$. Hence $P$ is  $\RR$-clean, which is a contradiction. 
\end{proof}

\cref{4regular} says that for a graph class to admit bounded blocking partitions, some structural assumption  in addition to bounded degree is necessary. \cref{BlockingPlanar} shows that bounded Euler genus is such an assumption.  We now show that bounded treewidth is another such assumption.

\begin{thm}\label{BlockingPartitionTW}
    Every graph \(G\) has a $2$-blocking partition with width at most \[1350 (\tw(G)+1)(\Delta(G))^2.\]
\end{thm}

The proof of \cref{BlockingPartitionTW} relies on a new lemma concerning tree-partitions. We say that a rooted $T$-partition $(B_x\colon x\in V(T))$ of a graph $G$ is \defn{detached} if for every non-root node $y\in V(T)$ with parent $x\in V(T)$, each vertex in $B_y$ is adjacent to at most one component of $G[B_x]$. 

\begin{lem}
\label{ImprovedTreeProduct}
Every graph $G$ has a detached $T$-partition of width at most \(90(\tw(G)+1)\Delta(G)\), for some tree $T$ with $\Delta(T)\leq 15\Delta(G)$
\end{lem}

The proof of \cref{ImprovedTreeProduct} builds on a clever argument due to a referee of a paper by \citet{DO95} showing that graphs with bounded treewidth and bounded maximum degree have tree-partitions of bounded width (see also \citep{Wood09,DW24}); see \cref{AppendixTreePartitions} for the details.

\begin{proof}[Proof of \cref{BlockingPartitionTW}]
     By \cref{ImprovedTreeProduct}, $G$ has a detached $T$-partition \((B_x : x \in V(T))\) with width at most $90(\tw(G)+1)\Delta(G)$ for some tree $T$ with $\Delta(T)\leq 15\Delta(G)$ and root $z\in V(G)$. Let $(V_0,V_1,\dots)$ be a \textsc{bfs}-layering of $G$ where $V_0=\{z\}$. We say a part $B_x$ is in level $i$ if $x\in V_i$. Colour the edges of $G$ as follows: each edge with two ends in one part \(B_x\) is coloured red, and each edge with one end in a part \(B_x\) at level \(i\) and one end in a part \(B_y\) at level \(i+1\) is coloured red if \(i\) is odd and blue if \(i\) is even. Let $\RR$ be the connected partition of $G$ where each part is the vertex-set of a component of the spanning subgraph of $G$ consisting of the red edges. Observe that the vertex-sets of the components of $G[B_z]$ are in $\RR$. Moreover, for every other part $X\in \RR$, there is a node $x\in V(T)$ with children $y_1,\dots,y_{\deg_T(x)-1}\in V(T)$ such that $X\subseteq B_x\cup B_{y_1} \cup \dots \cup B_{y_{\deg_T(x)-1}}$.  Since every node in $V(T)\setminus\{z\}$ has at most $15 \Delta(G)-1$ children, it follows that each part in $\RR$ has at most \((15 \Delta(G)) \cdot (90(\tw(G)+1)\Delta(G))\leq 1350 (\tw(G)+1)(\Delta(G))^2\) vertices. 
     
    For the sake of contradiction, suppose $G$ contains an $\RR$-clean path $P$ of length at least $3$. Since $P$ is $\RR$-clean, its edges are blue and so all edges of \(P\) are between levels \(i\) and \(i+1\) for some even \(i\), and all the vertices of \(P\) at level \(i\) belong to one part \(B_x\). Since each edge of $P$ is blue, the vertices of $P$ alternate between vertices in $B_x$ and vertices that belong to parts that are indexed by the children of $x$. Since $P$ has length at least $3$, $P$ has an internal vertex $w$ that belongs to $B_y$ for some child $y$ of $z$. Since $P$ is $\RR$-clean, $w$ is adjacent to at least two components of $G[B_x]$, contradicting \((B_x : x \in V(T))\) being a detached tree-partition. So every $\RR$-clean path in $G$ has length at most $2$, as required.
\end{proof}

\section{Open Problems}
\label{OpenProblems}

We conclude with some open problems.

\begin{open}
What is the minimum integer $\ell$ for which there exists a function $f$ such that every planar graph $G$ has an $\ell$-blocking partition with width at most $f(\Delta(G))$? We have proved that $\ell\leq \BlockingNumber$, although we have chosen to simplify our proof rather than optimise the constant. 
\end{open}

\begin{open}
 Can \cref{ImprovedShallowMinorsPlanar} be proved with $f$ bounded by a polynomial function of $d,r,s$? Our proof gives \(f(d, r,s) \le (sd)^{O(r!)}\).
\end{open}

Consider the following open problems for $k$-planar graphs. 

\begin{open}
\label{kPlanarMinc}
What is the minimum integer $c$ such that there is a function $f$ for which every $k$-planar graph $G$ is contained in $H\boxtimes P \boxtimes K_{f(k)}$where $\tw(H)\leq c$? We know that $3\leq c\leq \TreewidthBound$.
\end{open}

\begin{open}
\label{kPlanarPoly}
Is there a constant $c$ and a polynomial function $f$ such that every $k$-planar graph $G$ is contained in $H\boxtimes P \boxtimes K_{f(k)}$ where $\tw(H)\leq c$? 
Our proof gives \(f(k) \le 2^{O(\lfloor k/2\rfloor!)}\).
\end{open}

Questions analogous to \cref{kPlanarMinc,kPlanarPoly} can be asked for other natural classes.

Finally, consider what other graph classes have an $\ell$-blocking partitions?

\begin{open}
Does there exist integers $\ell,c\geq 1$ such that every graph with maximum degree at most \(3\) has an \(\ell\)-blocking partition of width at most $c$?
\end{open}

\begin{open}
For every $t\in \NN$, does there exist $k_t\in \NN$ and a function $f_t$ such that every $K_t$-minor-free graph $G$ has a $k_t$-blocking partition with width at most $f_t(\Delta(G))$? 
\end{open}

{\fontsize{10pt}{11pt}\selectfont
\bibliographystyle{DavidNatbibStyle}
\bibliography{DavidBibliography}}

\def\soft#1{\leavevmode\setbox0=\hbox{h}\dimen7=\ht0\advance \dimen7
  by-1ex\relax\if t#1\relax\rlap{\raise.6\dimen7
  \hbox{\kern.3ex\char'47}}#1\relax\else\if T#1\relax
  \rlap{\raise.5\dimen7\hbox{\kern1.3ex\char'47}}#1\relax \else\if
  d#1\relax\rlap{\raise.5\dimen7\hbox{\kern.9ex \char'47}}#1\relax\else\if
  D#1\relax\rlap{\raise.5\dimen7 \hbox{\kern1.4ex\char'47}}#1\relax\else\if
  l#1\relax \rlap{\raise.5\dimen7\hbox{\kern.4ex\char'47}}#1\relax \else\if
  L#1\relax\rlap{\raise.5\dimen7\hbox{\kern.7ex
  \char'47}}#1\relax\else\message{accent \string\soft \space #1 not
  defined!}#1\relax\fi\fi\fi\fi\fi\fi}
\begin{thebibliography}{40}
\providecommand{\natexlab}[1]{#1}
\providecommand{\msn}[1]{MR:\,\href{http://www.ams.org/mathscinet-getitem?mr=MR{#1}}{#1}}
\providecommand{\ZBL}[1]{Zbl:\,\href{https://www.zentralblatt-math.org/zmath/en/search/?q=an:#1}{#1}}
\providecommand{\url}[1]{\texttt{#1}}
\providecommand{\urlprefix}{}
\expandafter\ifx\csname urlstyle\endcsname\relax
  \providecommand{\doi}[1]{doi:\discretionary{}{}{}#1}\else
  \providecommand{\doi}{doi:\discretionary{}{}{}\begingroup
  \urlstyle{rm}\Url}\fi

\bibitem[{Angelini et~al.(2018)Angelini, Bekos, Kaufmann, Kindermann, and
  Schneck}]{ABKKS18}
\textsc{Patrizio Angelini, Michael~A. Bekos, Michael Kaufmann, Philipp
  Kindermann, and Thomas Schneck}.
\newblock \href{https://doi.org/10.1016/j.tcs.2018.03.005}{1-fan-bundle-planar
  drawings of graphs}.
\newblock \emph{Theoret. Comput. Sci.}, 723:23--50, 2018.

\bibitem[{Bonnet et~al.(2022)Bonnet, Kwon, and Wood}]{BKW}
\textsc{\'Edouard Bonnet, {O-joung} Kwon, and David~R. Wood}.
\newblock \href{https://arxiv.org/abs/2202.11858}{Reduced bandwidth: a
  qualitative strengthening of twin-width in minor-closed classes (and
  beyond)}.
\newblock 2022, arXiv:2202.11858.

\bibitem[{Bose et~al.(2020)Bose, Dujmovi\'c, Javarsineh, and Morin}]{BDJM}
\textsc{Prosenjit Bose, Vida Dujmovi\'c, Mehrnoosh Javarsineh, and Pat Morin}.
\newblock \href{https://arxiv.org/abs/2007.06455}{Asymptotically optimal vertex
  ranking of planar graphs}.
\newblock 2020, arXiv:2007.06455.

\bibitem[{Campbell et~al.(2024)Campbell, Clinch, Distel, Gollin, Hendrey,
  Hickingbotham, Huynh, Illingworth, Tamitegama, Tan, and Wood}]{UTW}
\textsc{Rutger Campbell, Katie Clinch, Marc Distel, J.~Pascal Gollin, Kevin
  Hendrey, Robert Hickingbotham, Tony Huynh, Freddie Illingworth, Youri
  Tamitegama, Jane Tan, and David~R. Wood}.
\newblock \href{https://doi.org/10.1017/S0963548323000457}{Product structure of
  graph classes with bounded treewidth}.
\newblock \emph{Combin. Probab. Comput.}, 33(3):351--376, 2024.

\bibitem[{Didimo et~al.(2019)Didimo, Liotta, and Montecchiani}]{DLM19}
\textsc{Walter Didimo, Giuseppe Liotta, and Fabrizio Montecchiani}.
\newblock \href{https://doi.org/10.1145/3301281}{A survey on graph drawing
  beyond planarity}.
\newblock \emph{{ACM} Comput. Surv.}, 52(1):4:1--4:37, 2019.

\bibitem[{Diestel(2018)}]{Diestel5}
\textsc{Reinhard Diestel}.
\newblock Graph theory, vol. 173 of \emph{Graduate Texts in Mathematics}.
\newblock Springer, 5th edn., 2018.

\bibitem[{Ding and Oporowski(1995)}]{DO95}
\textsc{Guoli Ding and Bogdan Oporowski}.
\newblock \href{https://doi.org/10.1002/jgt.3190200412}{Some results on tree
  decomposition of graphs}.
\newblock \emph{J. Graph Theory}, 20(4):481--499, 1995.

\bibitem[{Distel et~al.(2022)Distel, Hickingbotham, Huynh, and Wood}]{DHHW22}
\textsc{Marc Distel, Robert Hickingbotham, Tony Huynh, and David~R. Wood}.
\newblock \href{https://doi.org/10.48550/arXiv.2112.10025}{Improved product
  structure for graphs on surfaces}.
\newblock \emph{Discrete Math. Theor. Comput. Sci.}, 24(2):\#6, 2022.

\bibitem[{Distel and Wood(2024)}]{DW24}
\textsc{Marc Distel and David~R. Wood}.
\newblock \href{https://doi.org/10.1007/978-3-031-47417-0_11}{Tree-partitions
  with bounded degree trees}.
\newblock In \textsc{David~R. Wood, Jan de~Gier, and Cheryl~E. Praeger}, eds.,
  \emph{2021--2022 MATRIX Annals}, pp. 203--212. Springer, 2024.

\bibitem[{D\k{e}bski et~al.(2021)D\k{e}bski, Felsner, Micek, and
  Schr\"{o}der}]{DFMS21}
\textsc{Micha{\l} D\k{e}bski, Stefan Felsner, Piotr Micek, and Felix
  Schr\"{o}der}.
\newblock \href{https://doi.org/10.19086/aic.27351}{Improved bounds for
  centered colorings}.
\newblock \emph{Adv. Comb.}, \#8, 2021.

\bibitem[{Dujmovi\'c et~al.(2017)Dujmovi\'c, Eppstein, and Wood}]{DEW17}
\textsc{Vida Dujmovi\'c, David Eppstein, and David~R. Wood}.
\newblock \href{https://doi.org/10.1137/16M1062879}{Structure of graphs with
  locally restricted crossings}.
\newblock \emph{SIAM J. Discrete Math.}, 31(2):805--824, 2017.

\bibitem[{Dujmovi\'c et~al.(2021)Dujmovi\'c, Esperet, Gavoille, Joret, Micek,
  and Morin}]{DEGJMM21}
\textsc{Vida Dujmovi\'c, Louis Esperet, Cyril Gavoille, Gwena\"el Joret, Piotr
  Micek, and Pat Morin}.
\newblock \href{https://doi.org/10.1145/3477542}{Adjacency labelling for planar
  graphs (and beyond)}.
\newblock \emph{J. ACM}, 68(6):42, 2021.

\bibitem[{Dujmovi{\'c} et~al.(2020{\natexlab{a}})Dujmovi{\'c}, Esperet, Joret,
  Walczak, and Wood}]{DEJWW20}
\textsc{Vida Dujmovi{\'c}, Louis Esperet, Gwena\"{e}l Joret, Bartosz Walczak,
  and David~R. Wood}.
\newblock \href{https://doi.org/10.19086/aic.12100}{Planar graphs have bounded
  nonrepetitive chromatic number}.
\newblock \emph{Adv. Comb.}, \#5, 2020{\natexlab{a}}.

\bibitem[{Dujmovi\'c et~al.(2023)Dujmovi\'c, Hickingbotham, Hodor, Joret, La,
  Micek, Morin, Rambaud, and Wood}]{DHHJLMMRW}
\textsc{Vida Dujmovi\'c, Robert Hickingbotham, Jędrzej Hodor, Gwena\"el Joret,
  Hoang La, Piotr Micek, Pat Morin, Clément Rambaud, and David~R. Wood}.
\newblock \href{https://doi.org/10.1137/1.9781611977912.48}{The grid-minor
  theorem revisited}.
\newblock In \emph{Proc. 2024 Annual ACM-SIAM Symposium on Discrete Algorithms
  \textup{(SODA '24)}}, pp. 1241--1245. 2023.
\newblock arXiv:2307.02816.

\bibitem[{Dujmovi{\'c} et~al.(2020{\natexlab{b}})Dujmovi{\'c}, Joret, Micek,
  Morin, Ueckerdt, and Wood}]{DJMMUW20}
\textsc{Vida Dujmovi{\'c}, Gwena\"{e}l Joret, Piotr Micek, Pat Morin, Torsten
  Ueckerdt, and David~R. Wood}.
\newblock \href{https://doi.org/10.1145/3385731}{Planar graphs have bounded
  queue-number}.
\newblock \emph{J. ACM}, 67(4):\#22, 2020{\natexlab{b}}.

\bibitem[{Dujmovi{\'c} et~al.(2017)Dujmovi{\'c}, Morin, and Wood}]{DMW17}
\textsc{Vida Dujmovi{\'c}, Pat Morin, and David~R. Wood}.
\newblock \href{https://doi.org/10.1016/j.jctb.2017.05.006}{Layered separators
  in minor-closed graph classes with applications}.
\newblock \emph{J. Combin. Theory Ser. B}, 127:111--147, 2017.

\bibitem[{Dujmovi{\'c} et~al.(2023)Dujmovi{\'c}, Morin, and Wood}]{DMW23}
\textsc{Vida Dujmovi{\'c}, Pat Morin, and David~R. Wood}.
\newblock \href{https://doi.org/10.1016/j.jctb.2023.03.004}{Graph product
  structure for non-minor-closed classes}.
\newblock \emph{J. Combin. Theory Ser. B}, 162:34--67, 2023.

\bibitem[{Dujmović et~al.(2024)Dujmović, Hickingbotham, Joret, Micek, Morin,
  and Wood}]{DHJMMW24}
\textsc{Vida Dujmović, Robert Hickingbotham, Gwenaël Joret, Piotr Micek, Pat
  Morin, and David~R. Wood}.
\newblock \href{https://doi.org/10.1017/S0963548323000275}{The excluded tree
  minor theorem revisited}.
\newblock \emph{Combin. Probab. Comput.}, 33(1):85--90, 2024.

\bibitem[{Dvor{\'{a}}k et~al.(2022)Dvor{\'{a}}k, Gon{\c{c}}alves, Lahiri, Tan,
  and Ueckerdt}]{DGLTU22}
\textsc{Zdenek Dvor{\'{a}}k, Daniel Gon{\c{c}}alves, Abhiruk Lahiri, Jane Tan,
  and Torsten Ueckerdt}.
\newblock \href{https://doi.org/10.4230/LIPIcs.SoCG.2022.38}{On comparable box
  dimension}.
\newblock In \textsc{Xavier Goaoc and Michael Kerber}, eds., \emph{Proc. 38th
  Int'l Symp. on Computat. Geometry \textup{(SoCG 2022)}}, vol. 224 of
  \emph{LIPIcs}, pp. 38:1--38:14. Schloss Dagstuhl, 2022.

\bibitem[{Erd\H{o}s and Sachs(1963)}]{ES63}
\textsc{Paul Erd\H{o}s and Horst Sachs}.
\newblock Regul\"are {G}raphen gegebener {T}aillenweite mit minimaler
  {K}notenzahl.
\newblock \emph{Wiss. Z. Martin-Luther-Univ. Halle-Wittenberg Math.-Natur.
  Reihe}, 12:251--257, 1963.

\bibitem[{Esperet et~al.(2023)Esperet, Joret, and Morin}]{EJM23}
\textsc{Louis Esperet, Gwena\"{e}l Joret, and Pat Morin}.
\newblock \href{https://doi.org/10.1112/jlms.12781}{Sparse universal graphs for
  planarity}.
\newblock \emph{J. London Math. Soc.}, 108(4):1333--1357, 2023.

\bibitem[{Fox and Pach(2010)}]{FP10}
\textsc{Jacob Fox and J{\'a}nos Pach}.
\newblock \href{https://doi.org/10.1017/S0963548309990459}{A separator theorem
  for string graphs and its applications}.
\newblock \emph{Combin. Probab. Comput.}, 19(3):371--390, 2010.

\bibitem[{Fox and Pach(2012)}]{FP12}
\textsc{Jacob Fox and J{\'a}nos Pach}.
\newblock \href{https://doi.org/10.1016/j.aim.2012.03.011}{String graphs and
  incomparability graphs}.
\newblock \emph{Adv. Math.}, 230(3):1381--1401, 2012.

\bibitem[{Grigoriev and Bodlaender(2007)}]{GB07}
\textsc{Alexander Grigoriev and Hans~L. Bodlaender}.
\newblock \href{https://doi.org/10.1007/s00453-007-0010-x}{Algorithms for
  graphs embeddable with few crossings per edge}.
\newblock \emph{Algorithmica}, 49(1):1--11, 2007.

\bibitem[{Hickingbotham and Wood(2024)}]{HW24}
\textsc{Robert Hickingbotham and David~R. Wood}.
\newblock \href{https://doi.org/10.1137/22M1540296}{Shallow minors, graph
  products and beyond-planar graphs}.
\newblock \emph{SIAM J. Discrete Math.}, 38(1):1057--1089, 2024.

\bibitem[{Hong and Tokuyama(2020)}]{HT20}
\textsc{Seok{-}Hee Hong and Takeshi Tokuyama}, eds.
\newblock \href{https://doi.org/10.1007/978-981-15-6533-5}{Beyond planar
  graphs}.
\newblock Springer, 2020.

\bibitem[{Illingworth et~al.(2022)Illingworth, Scott, and Wood}]{ISW}
\textsc{Freddie Illingworth, Alex Scott, and David~R. Wood}.
\newblock \href{https://arxiv.org/abs/2104.06627}{Product structure of graphs
  with an excluded minor}.
\newblock 2022, arXiv:2104.06627.
\newblock \emph{Trans. Amer. Math. Soc.}, to appear.

\bibitem[{Jacob and Pilipczuk(2022)}]{JP22}
\textsc{Hugo Jacob and Marcin Pilipczuk}.
\newblock \href{https://doi.org/10.1007/978-3-031-15914-5\_21}{Bounding
  twin-width for bounded-treewidth graphs, planar graphs, and bipartite
  graphs}.
\newblock In \textsc{Michael~A. Bekos and Michael Kaufmann}, eds., \emph{Proc.
  48th International Workshop on Graph-Theoretic Concepts in Computer Science
  \textup{({WG} 2022})}, vol. 13453 of \emph{Lecture Notes in Comput. Sci.},
  pp. 287--299. Springer, 2022.

\bibitem[{Kráľ et~al.(2023)Kráľ, Pekárková, and Štorgel}]{KPS23}
\textsc{Daniel Kráľ, Kristýna Pekárková, and Kenny Štorgel}.
\newblock \href{http://arxiv.org/abs/2307.05811}{Twin-width of graphs on
  surfaces}.
\newblock 2023, arXiv:2307.05811.

\bibitem[{Matou{\v{s}}ek(2014)}]{Mat14}
\textsc{Ji{\v{r}}{\'{\i}} Matou{\v{s}}ek}.
\newblock \href{https://doi.org/10.1017/S0963548313000400}{Near-optimal
  separators in string graphs}.
\newblock \emph{Combin. Probab. Comput.}, 23(1):135--139, 2014.

\bibitem[{Mohar and Thomassen(2001)}]{MoharThom}
\textsc{Bojan Mohar and Carsten Thomassen}.
\newblock Graphs on surfaces.
\newblock Johns Hopkins University Press, 2001.

\bibitem[{Ne{\v{s}}et{\v{r}}il and Ossona De~Mendez(2008)}]{NesOdM-GradI}
\textsc{Jaroslav Ne{\v{s}}et{\v{r}}il and Patrice Ossona De~Mendez}.
\newblock \href{https://doi.org/10.1016/j.ejc.2006.07.013}{Grad and classes
  with bounded expansion {I}. {D}ecompositions}.
\newblock \emph{European J. Combin.}, 29(3):760--776, 2008.

\bibitem[{Ossona~de Mendez(2021)}]{OdM-Banff}
\textsc{Patrice Ossona~de Mendez}.
\newblock Product structure for squares of planar graphs.
\newblock 2021.
\newblock Open Problems for Workshop on Graph Product Structure Theory, Banff
  International Research Station (21w5235).

\bibitem[{Pach and T{\'o}th(1997)}]{PachToth97}
\textsc{J{\'a}nos Pach and G{\'e}za T{\'o}th}.
\newblock \href{https://doi.org/10.1007/BF01215922}{Graphs drawn with few
  crossings per edge}.
\newblock \emph{Combinatorica}, 17(3):427--439, 1997.

\bibitem[{Pach and T{\'o}th(2002)}]{PachToth-DCG02}
\textsc{J{\'a}nos Pach and G{\'e}za T{\'o}th}.
\newblock \href{https://doi.org/10.1007/s00454-002-2891-4}{Recognizing string
  graphs is decidable}.
\newblock \emph{Discrete Comput. Geom.}, 28(4):593--606, 2002.

\bibitem[{Pilipczuk and Siebertz(2021)}]{PS21}
\textsc{Micha{\l} Pilipczuk and Sebastian Siebertz}.
\newblock \href{https://doi.org/10.1016/j.jctb.2021.06.002}{Polynomial bounds
  for centered colorings on proper minor-closed graph classes}.
\newblock \emph{J. Combin. Theory Ser. B}, 151:111--147, 2021.

\bibitem[{Robertson and Seymour(1986)}]{RS-II}
\textsc{Neil Robertson and Paul Seymour}.
\newblock \href{https://doi.org/10.1016/0196-6774(86)90023-4}{Graph minors.
  {II}. {A}lgorithmic aspects of tree-width}.
\newblock \emph{J. Algorithms}, 7(3):309--322, 1986.

\bibitem[{Schaefer and {\v{S}}tefankovi{\v{c}}(2004)}]{SS-JCSS04}
\textsc{Marcus Schaefer and Daniel {\v{S}}tefankovi{\v{c}}}.
\newblock \href{https://doi.org/10.1016/j.jcss.2003.07.002}{Decidability of
  string graphs}.
\newblock \emph{J. Comput. System Sci.}, 68(2):319--334, 2004.

\bibitem[{van~den Heuvel et~al.(2017)van~den Heuvel, Ossona~de Mendez, Quiroz,
  Rabinovich, and Siebertz}]{HOQRS17}
\textsc{Jan van~den Heuvel, Patrice Ossona~de Mendez, Daniel Quiroz, Roman
  Rabinovich, and Sebastian Siebertz}.
\newblock \href{https://doi.org/10.1016/j.ejc.2017.06.019}{On the generalised
  colouring numbers of graphs that exclude a fixed minor}.
\newblock \emph{European J. Combin.}, 66:129--144, 2017.

\bibitem[{Wood(2009)}]{Wood09}
\textsc{David~R. Wood}.
\newblock \href{https://doi.org/10.1016/j.ejc.2008.11.010}{On
  tree-partition-width}.
\newblock \emph{European J. Combin.}, 30(5):1245--1253, 2009.

\end{thebibliography}

\appendix

\section{Detached Tree-Partitions}\label{AppendixTreePartitions}

This appendix is devoted to the proof of \cref{ImprovedTreeProduct}. Recall that a rooted tree-partition $(B_x\colon x\in V(T))$ of a graph $G$ is \defn{detached} if for every non-root node $y\in V(T)$ with parent $x\in V(T)$, each vertex in $B_y$ is adjacent to at most one component of $G[B_x]$. 

\begin{lem}
\label{Detaching}
For any graph $G$, for any non-empty set $S\subseteq V(G)$, there exists a set $X$ such that:
\begin{itemize}
    \item $S\subseteq X\subseteq V(G)$;
    \item $|X|\leq 2|S|-1$; and
    \item each vertex in $G-X$ is adjacent to at most one component of $G[X]$.
\end{itemize}
\end{lem}

\begin{proof}
     Consider the following algorithm: 
Initialise $i:=0$ and $S_0:=S$. 
While there is a vertex $v$ in $G-S_i$ adjacent to at least two components of $G[S_i]$, let $S_{i+1}:=S_i\cup \{v\}$ and $i:=i+1$.

Say this algorithm stops at $i=m$. Let $X:=S_m$. Then each vertex in $G-X$ is adjacent to at most one component of $G[X]$. Let $c_j$ be the number of components of $G[S_j]$. By construction, 
$|S_j|=|S|+j$ and $c_j\leq c_0-j \leq|S|-j$ for each $j\in\{0,\dots,m\}$. In particular,  if $m\geq |S|-1$ then $c_{|S|-1}=1$. Thus $m\leq |S|-1$ and $|X|\leq |S|+m \leq 2|S|-1$.    
\end{proof}

The following lemma is the core of the proof of \cref{ImprovedTreeProduct}.

\begin{lem}
\label{heart}
For $k,d\in\NN$, for any graph $G$ with $\tw(G)\leq k-1$ and $\Delta(G)\leq d$, for any set $S\subseteq V(G)$ with $5 k\leq|S| \leq 30 kd$, there exists a detached tree-partition $(B_x:x\in V(T))$ of $G$ with root $z\in V(T)$ such that:
\begin{itemize}
    \item $\Delta(T)\leq 15d$;
    \item $|B_x|\leq 90kd $ for each $x\in V(T)$;
    \item $S\subseteq B_z$;
    \item $|B_z|\leq 3|S|-5k$; and
    \item $\deg_T(z)\leq \frac{|S|}{2k} - 1$.
\end{itemize}
\end{lem}

\begin{proof}
We proceed by induction on $|V(G)$|.

\textbf{Case 1.} $|V(G-S)|\leq 90kd$: Let $T$ be the tree with $V(T)=\{y,z\}$ and $E(T)=\{yz\}$. Note that $\Delta(T)=1\leq 15d$ and $\deg_T(z)=1\leq \frac{|S|}{2k}-1$. By \cref{Detaching}, there exists a set $B_z\subseteq V(G)$ such that $S\subseteq B_z$, $|B_z|\leq 2|S|-1\leq 3 |S|-5k\leq 90kd$ and every vertex in $V(G)-B_z$ is adjacent to at most two components of $G[B_z]$. Set $B_y:=V(G)-B_z$. Then $|B_y|\leq |V(G)-S|\leq 90kd$ and every vertex in $B_y$ is adjacent to at most one component of $G[B_z]$. Hence $(B_x:x\in V(T))$ is the desired detached tree-partition of $G$.

Now assume that $|V(G-S)|\geq 90kd$.

\textbf{Case 2.} $5k \leq |S|\leq 15 k$: By \cref{Detaching}, there exists a set $B_z\subseteq V(G)$ such that $S\subseteq B_z$, $|B_z|\leq 2|S|\leq \min\{3|S|-5k,30k\}$ and every vertex in $V(G)-B_z$ is adjacent to at most one component of $G[B_z]$. Let $S':=\bigcup\{ N_G(v)\setminus B_z: v\in B_z\}$. So $|S'|\leq d |B_z|\leq  30kd$. If
$|S'|< 5 k$ then add $5k-|S'|$ vertices from $V(G-B_z-S')$ to $S'$, so
that $|S'|=5k$. This is well-defined since $|V(G-B_z)| \geq 90kd-30k \geq 5k$, implying $|V(G-B_z-S')| \geq 5k-|S'|$. By induction, there exists a detached tree-partition $(B_x:x\in V(T'))$ of $G-B_z$ with root $z'\in V(T')$ such that:
\begin{itemize}
    \item $|B_x|\leq 90kd $ for each $x\in V(T')$;
    \item $\Delta(T')\leq 15d$;
    \item $S'\subseteq B_{z'}$;
    \item $|B_{z'}|\leq 3 |S'|-5k \leq 90kd$; and
    \item $\deg_{T'}(z')\leq \frac{|S'|}{2k} - 1 \leq 15d - 1$.
\end{itemize}
Let $T$ be the rooted tree obtained from $T'$ by adding a new root $z$ adjacent to $z'$. So $(B_x:x\in V(T))$ is a
tree-partition of $G$ with width at most
$\max\{90kd,|B_z|\}\leq\max\{90kd,30k\}=90kd$. By construction,
$\deg_T(z)=1 \leq \frac{|S|}{2k} - 1$ and $\deg_{T}(z') =
\deg_{T'}(z')+1\leq (15d-1) + 1 = 15d$. Every other vertex in $T$ has
the same degree as in $T'$. Hence $\Delta(T)\leq 15d$, as desired.
Finally, since $(B_x:x\in V(T'))$ is detached and every vertex in $V(G)-B_z$ is adjacent to at most one component of $G[B_z]$, it follows that $(B_x:x\in V(T))$ is also detached.

\textbf{Case 3.} $15 k \leq |S|\leq 30kd$: By the separator lemma of
\citet[(2.6)]{RS-II}, there are induced subgraphs $G_1$ and $G_2$ of
$G$ with $G_1\cup G_2=G$ and $|V(G_1\cap G_2)|\leq k$, where $|S\cap
V(G_i)|\leq \frac23 |S|$ for each $i\in\{1,2\}$. Let $S_i := (S\cap
V(G_i))\cup V(G_1\cap G_2)$ for each $i\in\{1,2\}$.

We now bound $|S_i|$. For a lower bound, since $|S\cap V(G_1)|\leq
\frac23 |S|$, we have $|S_2|\geq |S\setminus V(G_1)|\geq \frac13 |S|
\geq 5k $. By symmetry, $|S_1|\geq  5k $. For an upper bound,
$|S_i|\leq\frac23 |S| + k \leq 20kd + k \leq 30kd$. Also note that
$|S_1|+|S_2|\leq |S|+2k$.

We have shown that $5k \leq |S_i|\leq 30kd$ for each $i\in\{1,2\}$.
Thus we may apply induction to $G_i$ with $S_i$ the specified set.
Hence there exists a detached tree-partition $(B^i_x:x\in V(T_i))$ of $G_i$ with root $z_i\in V(T_i)$ such that:
\begin{itemize}
    \item $|B^i_x|\leq 90kd $ for each $x\in V(T_i))$;
    \item $\Delta(T_i)\leq 15d$;
    \item $S_i\subseteq B_{z_i}$;
    \item $|B_{z_i}|\leq 3|S_i|-5k$; and
    \item $\deg_{T_i}(z_i)\leq \frac{|S_i|}{2k}-1$.
\end{itemize}
Let $T$ be the rooted tree obtained from the disjoint union of $T_1$ and
$T_2$ by identifying $z_1$ and $z_2$ into a new root vertex $z$. Let $B_z:=
B^1_{z_1}\cup B^2_{z_2}$. Let $B_x:= B^i_x$ for each $x\in
V(T_i)\setminus\{z_i\}$. Since $G=G_1\cup G_2$ and $V(G_1\cap
G_2)\subseteq B^1_{z_1}\cap B^2_{z_2} \subseteq B_z$, we have that
$(B_x:x\in V(T))$ is a tree-partition of $G$.
By construction, $S\subseteq B_z$ and since $V(G_1\cap G_2)\subseteq
B^i_{z_i}$ for each $i$,
\begin{align*}
    |B_z|
    & \leq |B^1_{z_1}|+|B^2_{z_2}| - |V(G_1\cap G_2)|\\
    & \leq (3|S_1|-5k) +  (3|S_2|-5k) - |V(G_1\cap G_2)|\\
    & = 3( |S_1|+ |S_2|) -10k - |V(G_1\cap G_2)|\\
    & \leq 3( |S| + 2|V(G_1\cap G_2)| ) -10k - |V(G_1\cap G_2)|\\
    & \leq 3 |S| + 5|V(G_1\cap G_2)| -10k\\
    & \leq 3 |S| - 5 k\\
    & < 90kd.
\end{align*}
Every other part has the same size as in the tree-partition of $G_1$
or $G_2$. So this tree-partition of $G$ has width at most $90kd$.
Note that
\begin{align*}
 \deg_T(z)   = \deg_{T_1}(z_1) + \deg_{T_2}(z_2)
    & \leq  (\tfrac{|S_1|}{2k}-1) + (\tfrac{|S_2|}{2k}-1)\\
    & =  \tfrac{|S_1|+|S_2|}{2k} -2\\
    & \leq  \tfrac{|S|+2k}{2k} -2\\
    & =  \tfrac{|S|}{2k} -1\\
    & < 15 d.
\end{align*}
Every other node of $T$ has the same degree as in $T_1$ or $T_2$.
Thus $\Delta(T) \leq 15d$. So it remains to show that $(B_x \colon x\in V(T))$ is detached. By induction, it follows that for every node $x\in V(T)\setminus\{z\}$ with child $y$, every vertex in $B_y$ is adjacent to at most one component of $G[B_x]$. Now suppose that a vertex $v\in V(G)-B_z$ is adjacent to at least two components of $G[B_z]$. Let $u,w\in B_z$ be neighbours of $v$ in $G$ that belong to distinct components of $G[B_z]$. Since $(B^i_x\colon x\in V(T_i))$ is a detached tree-partition of $G_i$, it follows that either $u\in V(G_1)\setminus V(G_2)$ and $w\in V(G_2)\setminus V(G_1)$, or $u\in V(G_2)\setminus V(G_1)$ and $w\in V(G_1)\setminus V(G_2)$. As such, $v\in V(G_1)\cap V(G_2)$, but this is a contradiction since $V(G_1)\cap V(G_2)\subseteq B_z$. So $(B_x \colon x\in V(T))$ is detached, which completes the proof.
\end{proof}

\begin{proof}[Proof of \cref{ImprovedTreeProduct}]
First suppose that $|V(G)| < 5 (\tw(G)+1)$. Let $T$ be the 1-vertex tree with
$V(T)=\{x\}$, and let $B_x:=V(G)$. Then $(B_x:x\in V(T))$ is the
desired detached tree-partition, since $|B_x|=|V(G)|<5 (\tw(G)+1) \leq 90 (\tw(G)+1)\Delta(G)$ and
$\Delta(T)=0\leq 15\Delta(G)$.
Now assume that $|V(G)| \geq 5 (\tw(G)+1)$. The result follows from
\cref{heart}, where $S$ is any set of $5(\tw(G)+1)$ vertices in $G$.
\end{proof}
\end{document}